\documentclass{amsart}

          \usepackage{amssymb}
          \usepackage{amsmath}
          \usepackage{amsthm}
          \usepackage{amsfonts}
          \usepackage[english]{babel}
          \usepackage[utf8]{inputenc}
          \usepackage{enumerate}
          \usepackage{graphicx}
          \usepackage{url}
          \usepackage{color}

\usepackage[numbers,sort&compress]{natbib}

\newcommand{\comment}[1]{}

\newcommand{\ind}{{\bf 1}}

\def\inddd#1{{\ind}_{\left\{#1\right\}}}
\def\indn#1{\{#1_n\}_{n\in\N}}
\newcommand{\proba}{\mathbb P}
\newcommand{\esp}{{\mathbb E}}

\newcommand{\inv}{^{-1}}
\newcommand{\cov}{{\rm{Cov}}}
\newcommand{\var}{{\rm{Var}}}

\newcommand{\eqnh}{\begin{eqnarray*}}
\newcommand{\eqne}{\end{eqnarray*}}
\newcommand{\eqnhn}{\begin{eqnarray}}
\newcommand{\eqnen}{\end{eqnarray}}
\newcommand{\equh}{\begin{equation}}
\newcommand{\eque}{\end{equation}}

\def\summ#1#2#3{\sum_{#1 = #2}^{#3}}
\def\prodd#1#2#3{\prod_{#1 = #2}^{#3}}
\def\sif#1#2{\sum_{#1=#2}^\infty}

\newcommand{\eqd}{\stackrel{d}{=}}

\def\topp#1{^{(#1)}}

\def\abs#1{\left|#1\right|}

\def\ccbb#1{\left\{#1\right\}}

\def\pp#1{\left(#1\right)}
\def\spp#1{(#1)}
\def\bb#1{\left[#1\right]}

\def\mmid{\;\middle\vert\;}

\def\floor#1{\left\lfloor #1 \right\rfloor}
\def\sfloor#1{\lfloor #1 \rfloor}

\def\vv#1{{\boldsymbol #1}}
\def\vvi{{\boldsymbol i}}
\def\vvj{{\boldsymbol j}}
\def\vvk{{\boldsymbol k}}
\def\vvl{{\boldsymbol \ell}}
\def\vvm{{\boldsymbol m}}
\def\vvn{{\boldsymbol n}}

\def\vvs{{\boldsymbol s}}
\def\vvt{{\boldsymbol t}}
\def\vvH{{\boldsymbol H}}

\def\mand{\mbox{ and }}
\def\qmand{\quad\mbox{ and }\quad}

\def\qmwith{\quad\mbox{ with }\quad}
\def\mfa{\mbox{ for all }}

\def\mmas{\mbox{ as }}

\def\wt#1{\widetilde{#1}}
\def\wb#1{\overline{#1}}
\def\what#1{\widehat{#1}}

\def\limn{\lim_{n\to\infty}}
\def\limm{\lim_{m\to\infty}}

\def\weakto{\Rightarrow}

\def\Z{{\mathbb Z}}
\def\Zd{{\mathbb Z^d}}
\def\R{{\mathbb R}}

\def\N{{\mathbb N}}
\def\Nd{{\mathbb N^d}}
\def\B{{\mathbb B}}

\def\calA{\mathcal A}
\def\calC{\mathcal C}

\def\calF{\mathcal F}
\def\calG{\mathcal G}
\def\calN{\mathcal N}
\def\calP{\mathcal P}
\def\calS{\mathcal S}
\def\calT{\mathcal T}
\def\calY{\mathcal Y}
\def\calZ{\mathcal Z}

\def\topp#1{^{\scriptscriptstyle (#1)}}

\def\X{{\mathbb X}}

\newtheorem{Thm}{Theorem}[section]
\newtheorem{Lem}[Thm]{Lemma}
\newtheorem{Prop}[Thm]{Proposition}

\theoremstyle{definition}
\newtheorem{Rem}[Thm]{Remark}

\numberwithin{equation}{section}

\title[From random partitions to fractional Brownian sheets]{From random partitions  \\to fractional Brownian sheets}
\date{\today}

\author{Olivier Durieu}
\address{
Olivier Durieu\\
Laboratoire de Math\'ematiques et Physique Th\'eorique, UMR-CNRS 7350\\
F\'ed\'eration Denis Poisson, FR-CNRS 2964\\
Universit\'e de Tours, Parc de Grandmont, 37200 Tours, France.
}
\email{olivier.durieu@lmpt.univ-tours.fr}

\author{Yizao Wang}
\address
{
Yizao Wang\\
Department of Mathematical Sciences\\
University of Cincinnati\\
2815 Commons Way\\
Cincinnati, OH, 45221-0025, USA.
}
\email{yizao.wang@uc.edu}
\def\Zt{\mathbb Z^2}

\date{}

\begin{document}\sloppy

\begin{abstract}
We propose discrete random-field models that are based on random partitions of $\N^2$. 
The covariance structure of each random field is determined by the underlying random partition.   Functional central limit theorems are established for the proposed models, and fractional Brownian sheets, with full range of Hurst indices, arise in the limit. 
Our models could be viewed as  discrete analogues of fractional Brownian sheets, in the same spirit that the simple random walk is the discrete analogue of the Brownian motion.  
\end{abstract}

\maketitle

\section{Introduction}
In this paper, we propose random-field models that are based on random partitions, and show that their partial-sum random fields scale to fractional Brownian sheets. 
Our motivation came from three recent papers, one by \citet{hammond13power} and two by the authors and collaborator \citep{durieu16infinite,bierme17invariance}, where it was shown that fractional Brownian motions and some operator-scaling Gaussian random fields (that can be viewed as random-field generalizations of fractional Brownian motions, see \citep{bierme07operator}) may arise 
as the scaling limits of certain stochastic models, the dependence structure of which is essentially determined by certain random partitions of $\N=\{1,2,\ldots\}$. 
We start by briefly recalling the results in one dimension.

The two papers \citep{hammond13power,durieu16infinite} established functional central limit theorems for fractional Brownian motions based on two different random partitions. 
In each model, there is an underlying random partition of the integers $\{1,\dots,n\}$, and conditioning on the random partition, $\pm1$-valued random spins $X_1,\dots,X_n$ are assigned, in certain ways to be specified later. The advantage of taking random spins is that in this way, the covariances of the partial sums are determined by the underlying random partitions.  By appropriately choosing the random partition and the assignment rule of random spins, the partial sum $S_n = X_1+\cdots+X_n$ scales to a fractional Brownian motion as $n\to\infty$ in the form of
\equh\label{eq:fBm}
\ccbb{\frac{S_{\floor {nt}}}{n^HL(n)}}_{t\in[0,1]}\weakto \ccbb{\B^H_t}_{t\in[0,1]}
\eque
in $D([0,1])$ as $n\to\infty$,  
where
$L$ is a slowly varying function at infinity and 
 $\B^H$ on the right-hand side above denotes the fractional Brownian motion with Hurst index $H\in(0,1)$, a centered Gaussian process with covariance function
\[
\cov(\B^H_s,\B^H_t) = \frac12\pp{t^{2H}+s^{2H}-|t-s|^{2H}},\quad s,t\ge 0.
\]Throughout, $\weakto$ stands for convergence in distribution and $D([0,1])$ for the space of c\`adl\`ag functions equipped with the Skorohod topology \citep{billingsley99convergence}.

The models in \citep{hammond13power,durieu16infinite} are different in  both the underlying random partitions and the ways of assigning $\pm1$ spins, and they lead to different ranges of Hurst index: $H\in(0,1/2)$ in \citep{durieu16infinite} and $H\in(1/2,1)$ in \citep{hammond13power}.
The partial sum $S_n$ can be interpreted as a correlated random walk and provides a simple discrete counterpart to the fractional Brownian motion, in the same spirit that the simple random walk can be viewed as the discrete counterpart of the Brownian motion.

In view of the non-standard normalization $n^HL(n)$ in \eqref{eq:fBm} instead of $n^{1/2}$ for $S_n$ of the i.i.d.~random variables, such models are  sometimes referred to as having {\em long-range dependence} \citep{pipiras17long,samorodnitsky17stochastic}. 
Moreover,  the fractional Brownian motion in the limit characterizes the non-negligible dependence at macroscopic scale of the discrete model when $H\ne 1/2$ (recall that $\B^{1/2}$ is a standard Brownian motion). 
Such limit theorems are of special interest for the study of long-range dependence, as they often reveal different types of dynamics underlying certain common long-range dependence phenomena. Namely,  drastically different models may lead to the same stochastic process with long-range dependence, and fractional Brownian motions often show up in such limit theorems. 
Fractional Brownian motions, first considered by \citet{kolmogorov40wienersche} and studied rigorously by \citet{mandelbrot68fractional}, are arguably the most important stochastic processes in the investigation of long-range dependence: it is well known now that fractional Brownian motions arise in limit theorems on models from various areas, including finance \citep{kluppelberg04fractional}, telecommunications \citep{mikosch07scaling}, interacting particle systems \citep{peligrad08fractional}, 
aggregation of correlated random walks \citep{enriquez04simple}, just to mention a few. 
\medskip

Results in \citep{hammond13power,durieu16infinite} provide a new class of examples for long-range dependence: they may arise in the presence of certain random partitions. Such a point of view, to the best of our knowledge, has been rarely explored before.   Our motivation is to demonstrate that
the random-partition mechanism behind the long-range dependence phenomena in the aforementioned papers remains at work in a natural random-field setup. 
The generalization of aforementioned one-dimensional random partitions to high dimensions, however, is far from being unique. A first attempt has been successfully worked out in \citep{bierme17invariance}, where certain operator-scaling Gaussian random fields appear in the limit (see Remark \ref{rem:HS}).

Here we continue to explore other possibilities of random-field extensions. 
In particular, we shall propose three random-field extensions of the one-dimensional models, and show that the partial sums of proposed models scale to fractional Brownian sheets. Our limit theorems shall cover the full range of Hurst indices for the fractional Brownian sheets. 
This is in sharp contrast to the previous random-field model investigated earlier in \citep{bierme17invariance}, where the limit random fields are most of the time not fractional Brownian sheets.
This reflects the fact that the random partitions considered here are essentially different from the ones considered in \citep{bierme17invariance}, and hence  our models and limit theorems here complement the ones therein. 

Recall that a fractional Brownian sheet with Hurst index $\vvH = (H_1,H_2) \in(0,1)^2$ is a multi-parameter zero-mean Gaussian process $\{\B^\vvH_\vvt\}_{\vv t\in\R^2_+}$
 with covariance 
\[
\cov(\B^\vvH_\vvs,\B^\vvH_\vvt) = \prodd q12\frac12\pp{t_q^{2H_q} + s_q^{2H_q} - |t_q-s_q|^{2H_q}},\quad \vvt,\vvs\in\R^2_+.
\]
Fractional Brownian sheets are random-field generalizations of fractional Brownian motions proposed by \citet{kamont96fractional}. 
These are centered Gaussian processes that are operator-scaling (generalization of self-similarity to random fields, see e.g.~\citep{bierme07operator}) and with stationary rectangular increments. In the special case $H_1 =H_2 = 1/2$, the fractional Brownian sheet becomes the standard Brownian sheet, the random-field generalization of Brownian motion. For other Hurst indices,  fractional Brownian sheets exhibit anisotropic long-range dependence.  
Representation and path properties of these random fields have been extensively investigated. See for example the recent survey by \citet{xiao09sample}. Stochastic partial differential equations driven by fractional Brownian sheets have also been studied (see e.g.~\citep{hu00stochastic,oksendal01multiparameter}). 
At the same time, fractional Brownian sheets are not the only operator-scaling random fields with stationary rectangular increments. There are other random fields with long-range dependence which could also be viewed as generalization of fractional Brownian motions. Limit theorems for fractional Brownian sheets and other Gaussian random fields with long-range dependence, however, have not been as much developed as for fractional Brownian motions. 
Recent developments in this direction include for examples limit theorems for linear random fields \citep{lavancier07invariance,wang14invariance}, for set-indexed fields \citep{bierme14invariance}, and for aggregated models \citep{puplinskaite16aggregation,shen17operator}.

\medskip

Now we describe our models, which are extensions of the one-dimensional models in \citep{hammond13power,durieu16infinite} to two dimensions, in more details.
For these one-dimensional models, we first introduce a random partition of $\N$ into different components, where each component may have possibly an infinite number of elements. Next, given a random decomposition $\{\calC_k\}_{k\in\N}$ of $\N$, 
for each component $\calC$ we sample $X_\calC = \{X_i\}_{i\in\calC}$ according to a specific {\it assignment rule}, applied in an independent manner to all components $\{X_{\calC_k}\}_{k\in\N}$. For these models, each $X_i$ takes the values $\pm1$ only. We consider two possible assignment rules:
\medskip

\noindent {\it Identical assignment rule.} Assign the same values for all $X_i$ in the same component. The identical value is either $1$ or $-1$, with equal probabilities. 
\medskip

\noindent {\it Alternating assignment rule.} Assign $\pm1$ values in an alternating manner with respect to the natural order on $\N$, for $X_i$ in the same component. Given a component, there are two such ways of assigning $\pm1$ values, and one of them is chosen with probability $1/2$. For example, given a component $\calC = \{1,2,5\}$, the alternating assignment rule assigns $(1,-1,1)$ or $(-1,1,-1)$ with equal probabilities to $(X_1,X_2,X_5)$.
\medskip

In particular, the Hammond--Sheffield model in \citep{hammond13power} is based on a random partition of $\N$ induced by a random forest with infinite components, each being an infinite tree, and the identical assignment rule (the random forest induces actually a random partition of $\Z$). The model in \citep{durieu16infinite} is based on an exchangeable random partition on $\N$ \citep{pitman06combinatorial} induced  by a certain infinite urn scheme and the alternating assignment rule. It is a modification of a model originally investigated in \citet{karlin67central}, and hence we refer to the model as the {\em randomized Karlin model}. The two models will be recalled in full detail in later sections. Note that this framework of building stationary sequences based on random partitions and assignment rules also includes the example of independent $\pm1$ spins, of which the partial-sum process is well known to scale to a Brownian motion. To achieve this it suffices to take the finest partition of $\N$, that is, each component corresponds exactly to one element from $\N$, and then apply either assignment rule (the two are the same in this case).
\medskip

Our random-field models are based on random partitions of $\N^2$ obtained as the product of two independent random partitions of $\N$. Namely,
let $\calC\topp q = \{\calC_i\topp q\}_{i\in\N}, q=1,2$, be two partitions of $\N$. Let $\vv\calC = \calC\topp 1\times\calC\topp2$ denote the partition of $\N^2$ whose components are the Cartesian products $\calC_i\topp1\times\calC_j\topp2$ for all $i,j\in\N$ (e.g.~$\{1,2\}\times\{1,3\} = \{(1,1),(1,3),(2,1),(2,3)\}$ is a subset of $\N^2$).
Once the random partition is given, one of the two assignment rules is applied in each direction. Figure~\ref{fig:1} illustrates the product of partitions (left), the alternating assignment rule (middle), and an assignment rule of mixed type (right).

\begin{figure}[ht]
\includegraphics[width = 0.3\textwidth]{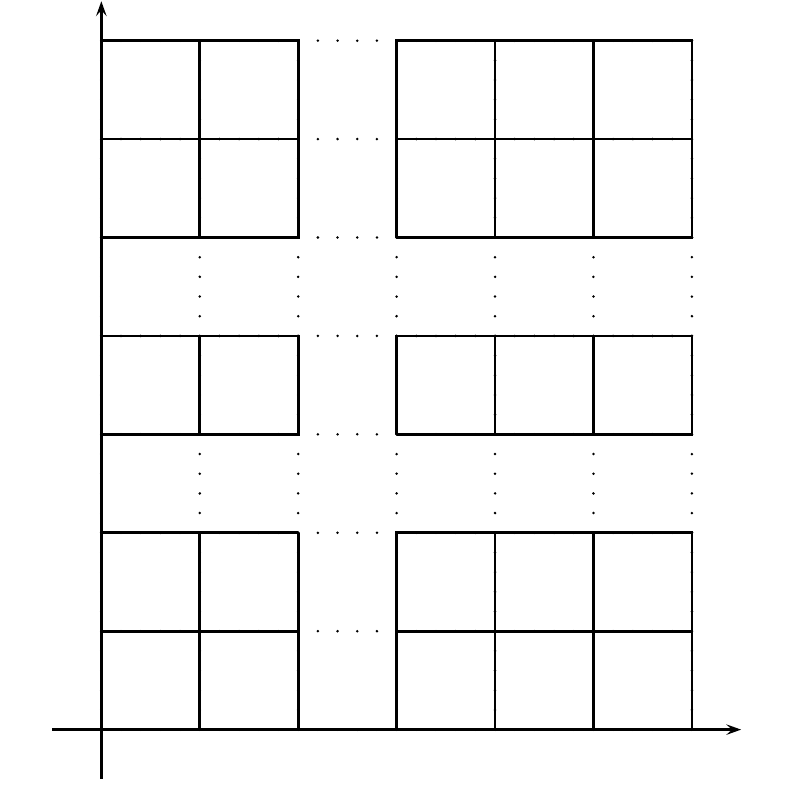}
\includegraphics[width = 0.3\textwidth]{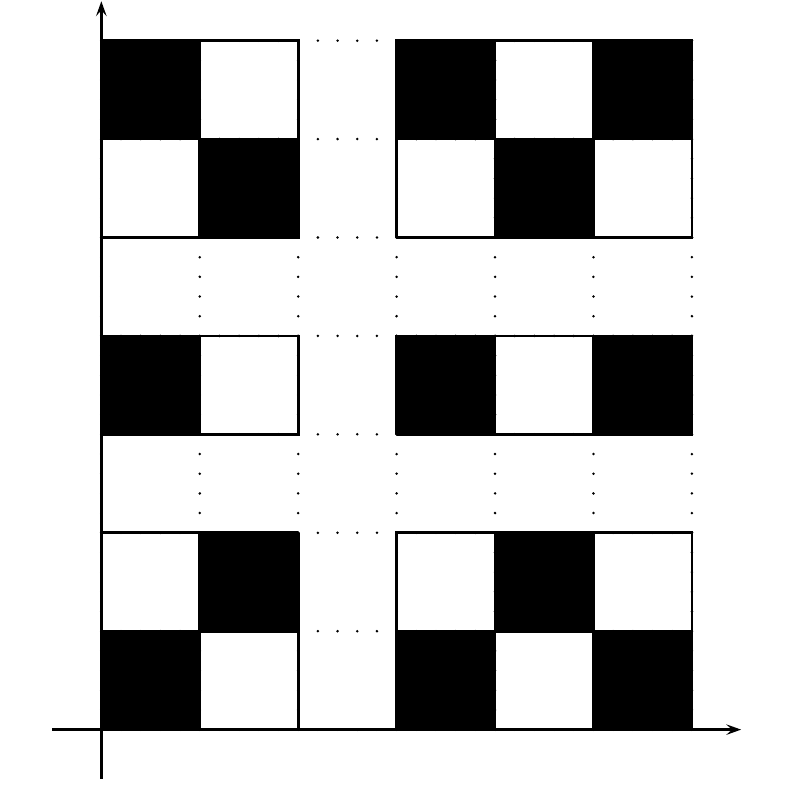}
\includegraphics[width = 0.3\textwidth]{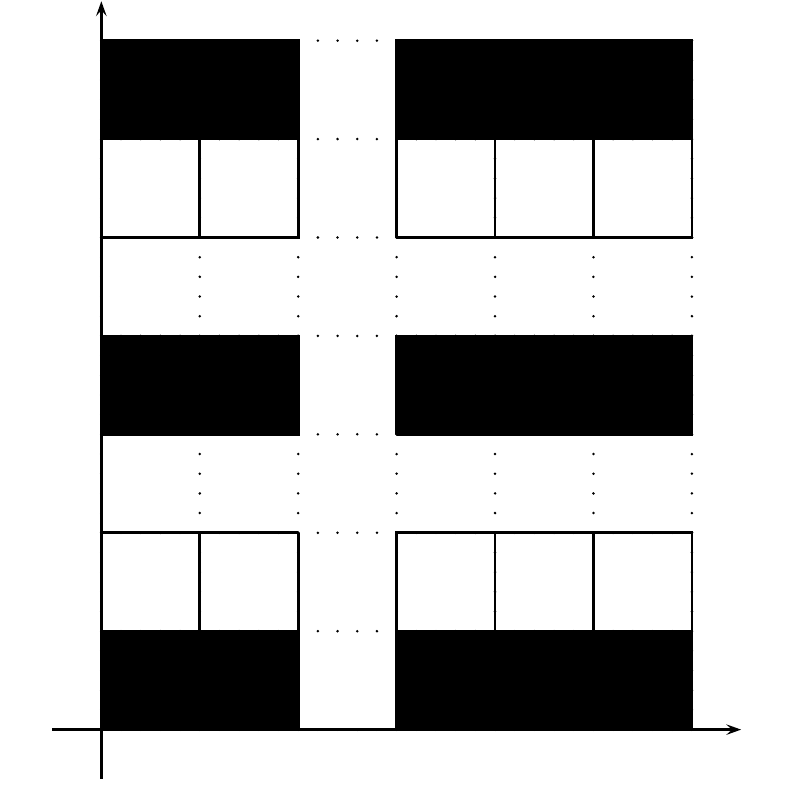}
\caption{\label{fig:1}A component from a product of partitions. Left: component $\{1,2,4,5,6\}\times\{1,2,4,6,7\}$. Middle: alternating assignment rule of $1$ (black) and $-1$ (white) values. Right: mixture of identical assignment rule in horizontal direction and alternating assignment rule in vertical direction.}
\end{figure}

We shall investigate the partial-sum random fields of three $\pm1$-valued models, each converging to fractional Brownian sheets in a different regime  in terms of the Hurst indices. The contributions of the paper are summarized here. 
\medskip

\noindent (i) In Section~\ref{sec:Karlin} we propose a generalization of the randomized Karlin model and show that the partial-sum random field scales to a fractional Brownian sheet with $\vv H\in(0,1/2)^2$. 
\medskip

\noindent (ii) In Section~\ref{sec:HS} we propose a generalization of the Hammond--Sheffield model and show that the partial-sum random field scales to a fractional Brownian sheet with $\vv H\in(1/2,1)^2$. 
\medskip
 
\noindent (iii) In Section~\ref{sec:combined} we propose a model that can be viewed as a combination of the  Hammond--Sheffield model and the randomized Karlin model, and show that the partial-sum random field scales to a  fractional Brownian sheet with $\vv H\in(1/2,1)\times(0,1/2)$.
\medskip

More specifically, our main results Theorems~\ref{thm:1},~\ref{thm:HSproduct}, and~\ref{thm:WIPcombined} are limit theorems in the form of
\[
\frac1{Z_\vvH(\vvn)}\ccbb{\sum_{\vvi \in [\vv1,\floor{\vvn\cdot \vvt}]}X_\vvi}_{\vvt\in[0,1]^2} \weakto 
\ccbb{\B_\vvt^\vvH}_{\vvt\in[0,1]^2}
\]
in $D([0,1]^2)$ as  $\min(n_1,n_2)\to\infty$, where $Z_\vvH(\vvn)$ is an appropriate normalization depending on the model (and hence on $\vvH$).
Our models can thus be viewed as discrete counterparts of fractional Brownian sheets.
Throughout, for any element $\vv a\in\R^2$, we write $\vv a = (a_1,a_2)$ and for any $\vv a, \vv b \in\Z^2$, we write $[\vv a,\vv b] = ([a_1,b_1]\times[a_2,b_2])\cap\Z^2$ the set of points in the rectangle $[a_1,b_1]\times[a_2,b_2]$ with integer coordinates.
We also use the notation $\floor{\vvn\cdot\vvt} = (\floor{n_1t_1},\floor{n_2t_2})$ where $\floor{\cdot}$ stands for the integer part.

\bigskip

The proofs for the two-dimensional randomized Karlin model and the other two models are completely different. For the randomized Karlin model, conditioning on the partition structure, the partial sums become sums of i.i.d.~random variables. 
For the other two models, the proof is based on a martingale central limit theorem due to \citet{mcleish74dependent}, already used for the one-dimensional Hammond--Sheffield model.
However, the proofs for two-dimensional models are much more demanding than their one-dimensional counterpart, as in general,  the martingale central limit theorem does not work as well for random fields as for stationary sequences, 
as pointed out by \citet{bolthausen82central} already in the 80s. Indeed, for the one-dimensional Hammond--Sheffield model, the normalized partial sum $S_n = X_1+\cdots+X_n$ 
can be expressed as
\[
\frac{S_n}{b_n} = \frac1{b_n}\sum_{j\in\Z}b_{n,j} X_j^*
\]
for a stationary sequence of martingale differences $\{X_j^*\}_{j\in\Z}$ and some coefficients $b_{n,j}$, with $b_n^2 = \sum_{j\in\Z}b_{n,j}^2$. 
This is a remarkable representation at the heart of the proof; see \citep{bierme17invariance} (the proof in \citep{hammond13power} did not use directly this convenient presentation, but applied nevertheless an argument by martingale approximation).
Then, to prove the weak convergence, by McLeish's central limit theorem, the key step is to show
\equh\label{eq:ergodicity1d}
\limn \frac1{b_n^2}\sum_{j\in\Z}b_{n,j}^2(X_i^*)^2 = \var (X_0^*)\mbox{ in probability}.
\eque
This requires already an involved argument in one dimension; see \citep[Lemma 3.2]{hammond13power} and \citep[Lemma 7]{bierme17invariance}. 

In two dimensions, the situation becomes even more complicated as now the partial sum $S_\vvn = \sum_{\vvi\in[\vv1,\vvn]}X_\vvi$ is expressed (Proposition~\ref{prop:representation} below) as 
\[
\frac{ S_{\vvn}}{b_\vvn} = \frac1{b_{n_1}\topp1}\sum_{j\in\Z}b_{n_1,j}\topp1 U_{j,n_2},
\]
where $\{U_{j,n_2}\}_{j\in\Z}$ is a stationary martingale-difference sequence with respect to the filtration corresponding to the first direction. 
The new difficulty of the random-field models comes from the fact that the martingale differences now also depend on $n_2$ and the dependence structure of the random partition in the second direction.  
To overcome the new difficulty, at the core of our proofs for the counterpart of \eqref{eq:ergodicity1d} is a decoupling argument. See Section \ref{sec:HSHS_CLT}.

\bigskip

We conclude the introduction with a few remarks.

\begin{Rem}
If one searches for a similar random field that scales to a fractional Brownian sheet with Hurst index $1/2$ in one direction, one can modify the proposed model by taking instead the finest partition (each integer consists of a component) in that  direction. Such a model and its analysis are much easier. The details are omitted. 
\end{Rem}

\begin{Rem}
It will become clear that our constructions are not limited to two dimensions only. Our limit theorems could also be extended accordingly to high dimensions, where the limit random fields cover fractional Brownian sheets with all legitimate indices. 
For high-dimensional models and the corresponding limit theorems, the analysis can be carried out by an induction argument, but will be notationally heavy.
Therefore, in this paper we focus on two dimensions, and only discuss the high dimension in Remark \ref{rem:high_dimension} at the end.
\end{Rem}

\begin{Rem}Our application of martingale central limit theorem is of a different nature from and much more complicated than the one for the other random-field extension in \citep{bierme17invariance} (see Remark~\ref{rem:HS}). 
There the partial sum can be expressed as a linear random field in the form of
\equh\label{eq:HS}
\frac{S_\vvn}{b_\vvn} = \frac1{b_\vvn}\sum_{\vvj\in\Z^d}b_{\vvn,\vvj}X_\vvj^*,\quad\vvn\in\Nd,
\eque
with a stationary sequence of martingale differences $\{X_\vvj^*\}_{\vvj\in\Zd}$ {\em in the lexicographical order} for all $d\in\N$, whence the analysis becomes {\em dimension-free}.  To the best of our knowledge, the model in \citep{bierme17invariance} is one of the very rare examples in the literature where a one-dimensional sequence of stationary martingale differences can be elegantly embedded into the presentation of the partial sums of a high-dimensional random field. In general,  embedding with respect to the lexicographical order could be formidable \citep{dedecker01exponential}, and the simple representation \eqref{eq:HS} seems rather a coincidence. 
\end{Rem}

\subsection*{Acknowledgements} 
The authors would like to thank two anonymous referees for their careful reading and helpful comments.
YW's research was partially supported by the NSA grants H98230-14-1-0318 and H98230-16-1-0322, the ARO grant W911NF-17-1-0006, and Charles Phelps Taft Research Center at University of Cincinnati.


\section{Randomized Karlin model}\label{sec:Karlin}
In this section we introduce the two-dimensional randomized Karlin model, and show that the partial-sum random field scales to a fractional Brownian sheet with Hurst index $\vvH\in(0,1/2)^2$.

\subsection{One-dimensional model}\label{sec:KarlinOne}
We first recall the one-dimensional randomized Karlin model \citep{karlin67central,durieu16infinite}. 
Let $\indn Y$ be i.i.d.~random variables with common distribution $\mu$ on $\N$. They induce a partition $\Pi_\infty$ of $\N$ by setting in the same equivalent class (component), denoted by $i\sim j$, if and only if $Y_i = Y_j$. 
Intuitively, imagine that we throw balls consecutively and independently into boxes labeled by $\N$, and set $Y_n = \ell$ if the $n$-th ball falls into the box with label $\ell$. This event occurs with probability $p_\ell=\mu(\{\ell\})$, and $i\sim j$ if and only if the balls at round $i$ and $j$ fall in the same box.
The partition obtained this way is an infinite exchangeable random partition of $\N$, sometimes referred to as the partition generated by random samplings, or the paintbox partition \citep{pitman06combinatorial}. 
Many estimates of this random partition that we apply here can be found in \citep{karlin67central,gnedin07notes}. 

Throughout, we assume that $\{p_\ell\}_{\ell\in\N}$ is a non-increasing sequence 
(this can always be assumed because the attached value of each label is irrelevant)
and that for some $\alpha\in(0,1)$,
\equh\label{eq:RV}
\nu(x)=\max\{\ell \ge 1 : p_\ell \ge 1/x\}\sim x^\alpha L(x),\mmas x\to\infty,
\eque
where $L$ is a slowly varying function at $\infty$. Without loss of generality $L$ is assumed to be bounded away from $0$ and $\infty$ on every compact set in $(0,\infty)$.
For example, the condition \eqref{eq:RV} is satisfied (with $L\equiv 1$) when
\[
p_k  \sim k^{-1/\alpha}, \mmas k\to\infty.
\]

The law of $(X_1,\dots,X_n)$ given the partition $\Pi_n$ of the set $\{1,\dots,n\}$ induced by $Y_1,\dots,Y_n$ is then determined by the alternating assignment rule. 
To express the alternating assignment rule explicitly, we introduce  
\equh\label{eq:Ynl}
Y_{n,\ell} = \summ i1n \inddd{Y_i = \ell},\;\ell \in \N,
\eque
 representing the number of balls in the box $\ell$ after the first $n$ sampling.
Then, the law of $(X_1,\dots,X_n)$ given the partition can be equivalently determined by letting $\indn\epsilon$ be i.i.d.~random variables, independent of $\Pi_\infty$, with common distribution $\proba(\epsilon_1 = -1) =\proba(\epsilon_1 = 1) = 1/2$ and setting for each $n\in\N$,
\equh\label{def:KarlinX_n}
X_n = \epsilon_\ell(-1)^{Y_{n,\ell}+1},\; \mbox{ if } Y_n = \ell.
\eque
Originally, \citet{karlin67central} obtained a central limit theorem for the non-randomized model, with $\epsilon_\ell \equiv 1$. The functional central limit theorem of the partial-sum process $S_n=X_1+\cdots+X_n$ 
was established in \citep{durieu16infinite}.

Later in Section~\ref{sec:combined}, we shall need a functional central limit theorem for a slightly more general version of  the one-dimensional Karlin model. 
We say that $\{X_n\}_{n\in\N}$ is a {\em generalized one-dimensional Karlin model}, if instead of \eqref{def:KarlinX_n} we have
\[
X_n = Z_\ell\epsilon_\ell(-1)^{Y_{n,\ell}+1},\; \mbox{ if } Y_n = \ell,
\]
for i.i.d.~random variables $\{Z_\ell\}_{\ell\in\N}$ with some common distribution $\nu$, independent from $Y$ and $\epsilon$. In this way, given the partition induced by $\{Y_n\}_{n\in\N}$, $X_i = X_j$ if $i\sim j$, and otherwise $X_i$ and $X_j$ are independent and identically distributed as $Z_1\epsilon_1$. 
\begin{Prop}\label{prop:Karlin1d}
For the generalized one-dimensional randomized Karlin model with $\mu$ satisfying~\eqref{eq:RV} with $\alpha\in(0,1)$ and a slowly varying function $L$, for a distribution $\nu$ with bounded support, we have
\[
\ccbb{\frac{S_{\floor {nt}}}{n^{\alpha/2} L(n)^{1/2}}}_{t\in[0,1]}\weakto \sigma_{\alpha}\ccbb{\B_{t}^{\alpha/2}}_{t\in[0,1]}
\]
in $D([0,1])$ as $n\to\infty$, with 
\[
\sigma_{\alpha}^2 =  \Gamma(1-\alpha)2^{\alpha-1}\var(X_1).
\]
\end{Prop}
The proof of this result is omitted here, as it can be obtained by following the same strategy as in the proof of Theorem \ref{thm:1}, the functional central limit theorem for two-dimensional randomized Karlin model to be introduced below. 
This result could also follow from a multivariate functional central limit theorem for the one-dimensional randomized Karlin model established in \citep[Theorem 2.2, Corollary 2.8]{durieu16infinite}, where the limit corresponds to
a decomposition of fractional Brownian motion into a bi-fractional Brownian motion and another smooth self-similar Gaussian process due to \citet{lei09decomposition}. In \citep{durieu16infinite}, only the randomized one-dimensional Karlin model was addressed,  although the same proof applies to the generalized model with $\nu$ having bounded support, too. 

\subsection{Two-dimensional model and main result}
We next describe the two-dimensional randomized Karlin model. For each $q=1,2$, consider $Y\topp q = \indn{Y\topp q}$ as i.i.d.~sampling from a certain probability measure $\mu_{q}$ satisfying~\eqref{eq:RV} with $\alpha_q\in(0,1)$ and a slowly varying function $L_q$. Assume that $Y\topp 1,Y\topp2$ are independent. Then, each $Y\topp q$ induces an infinite exchangeable random partition on $\N$ and for each $n$, let $Y\topp q_{n,\ell}$ be the corresponding statistics as in~\eqref{eq:Ynl} before. 
Write $Y_\vvn = (Y\topp1_{n_1}, 
Y\topp2_{n_2})$ and for every pair $\vv n, \vv m\in\N^2$, set
\[
\vv n\sim\vv m, \;\mbox{ if  } Y_\vvn = Y_\vvm.
\]
In this way,  equivalent subclasses (components) of $\N^2$ are indexed by labels $\vvl\in\N^2$. 
This gives the random partition of $\N^2$ as the product of the two partitions determined by $Y\topp 1$ and $Y\topp 2$.
Given the partition induced by $\{Y_\vvi\}_{\vvi\in[\vv1,\vvn]}$, the law of $\{X_{\vvi}\}_{\vv i\in[\vv1,\vvn]}$ is determined by  the alternating assignment rule in both directions. This is equivalent to set, letting $\{\epsilon_{\vvl}\}_{\vvl\in\N^2}$ be i.i.d.~random variables taking values in $\{-1,1\}$ with equal probabilities,
\[
X_{\vv n} =\epsilon_{\vvl}\prodd q12(-1)^{Y_{n_q,\ell_q}\topp q+1}, \;\mbox{ if } Y_{\vv n} = \vvl.
\]
The so-obtained random field $\{X_\vvn\}_{\vvn\in\N^2}$  is referred to as the two-dimensional randomized Karlin model. 

With a little abuse of language, for $\vvn\in\N^2$, we refer to $\{Y_\vvi\}_{\vvi\in[\vv1,\vvn]}$ as the {\it first $\vvn$ samplings}. We write $\vvn\to\infty$ for $\min(n_1,n_2)\to\infty$ and we write $\vvn^{\vv\alpha} = \prod_{q=1}^2 n_q^{\alpha_q}$ and $\vv L(\vv n)=\prod_{q=1}^2 L_q(n_q)$.

The main result of this section is the following.
\begin{Thm}\label{thm:1}
For the two-dimensional randomized Karlin model with $\mu_{q}$ satisfying~\eqref{eq:RV} with $\alpha_q\in(0,1)$ and slowly varying functions $L_q$ for $q=1,2$, we have
\[
\ccbb{\frac{S_{\floor {\vvn\cdot\vv t}}}{|\vvn|^{\vv\alpha/2} \vv L(\vvn)^{1/2}}}_{\vv t\in[0,1]^2}\weakto \sigma_{\vv\alpha}\ccbb{\B_{\vv t}^{\vv\alpha/2}}_{\vv t\in[0,1]^2}
\]
in $D([0,1]^2)$ as $\vvn\to\infty$, with $\sigma_{\vv\alpha}^2 = \prodd q12 \Gamma(1-\alpha_q)2^{\alpha_q-1}$.
\end{Thm}

\subsection{Auxiliary estimates}
Here we provide some useful estimates on the one-dimensional randomized Karlin model that we shall use in the proof of Theorem~\ref{thm:1}.
Recall $Y_{n,\ell}$ in~\eqref{eq:Ynl} and let
\[
K_n = \summ \ell1\infty \inddd{Y_{n,\ell}>0}\qmand
K_{n,r}  =  \summ \ell1\infty \inddd{Y_{n,\ell} = r},\;\mbox{ for all } r\in\N,
\]
denote the number of occupied boxes and number of boxes occupied with $r$ balls, respectively, after $n$ samplings. The statistics of $K_n$ and $K_{n,r}$ (independent from $\epsilon_\ell$) have been studied in \citep{karlin67central} already, where a similar model with $\epsilon_\ell$ replaced by constant $1$ was investigated.
We summarize some results on $K_n$ and $K_{n,r}$ below that will be needed later.  In the sequel, $\Gamma$ denotes the gamma function and for $r\ge1$ and $\alpha\in(0,1)$, we write
\[
p_\alpha(r) = \frac{\alpha(1-\alpha)\cdots(r-1-\alpha)}{r!}.
\]
Observe that $\sif r1p_\alpha(r) = 1$ and $\sif r1p_\alpha(2r-1) = 2^{\alpha-1}$.

\begin{Lem}\label{lem:K}
Under the assumption~\eqref{eq:RV}, we have
\eqnhn
\limn\frac{K_n}{n^\alpha L(n)} & = & \Gamma(1-\alpha),\label{eq:Kn}\\
\limn\frac{K_{n,r}}{n^\alpha L(n)} & = & \Gamma(1-\alpha)p_\alpha(r),\label{eq:Knr} \\
\limn\frac{\sif r1 K_{n,2r-1}}{n^\alpha L(n)} & = & \Gamma(1-\alpha) 2^{\alpha-1},\label{eq:Knodd}
\eqnen
where the convergences hold almost surely and also in $L^p$ for all $p>0$. 
\end{Lem}
\begin{proof}
(i) For the almost sure convergence in the three limits above, see \citep[Corollary 21 and discussion after Proposition 2]{gnedin07notes} and \citep[Theorem 9]{karlin67central}.

(ii) To prove the $L^p$ convergence, 
we prove~\eqref{eq:Kn} holds in $L^p$ for $p>0$. This and the facts that $0\leq K_{n,r}\leq K_n$ and $0\leq \sum_{r\geq 1}K_{n,2r-1}\leq K_n$ then imply the $L^p$ convergence in~\eqref{eq:Knr} and~\eqref{eq:Knodd}.  
For~\eqref{eq:Kn}, it suffices to prove the uniform integrability of the sequence $(K_n^p/(n^{\alpha}L(n))^p)_{n\ge1}$.
This follows, writing $\Phi_n=\esp K_n$, from the  asymptotic equivalence (see \citep[Proposition 17]{gnedin07notes})
\equh
\label{eq:Phi_n}
\Phi_n\sim \Gamma(1-\alpha)n^\alpha L(n), \mmas n\to\infty,
\eque 
the fact that for every $m\in\N$, there exist a constant $C_m$, such that for all $n\in\N$, 
\equh\label{eq:Knm}
\esp K_n^m \leq C_m\Phi_n^m,
\eque
and then an application of the de la Vall\'ee Poussin criterion for uniform integrability:  $(K_n/(n^\alpha L(n)))_{n\in\N}$ is bounded in $L^m$ for $m>p$.
To see~\eqref{eq:Knm}, we need the following lemma.
\begin{Lem}
For $n\in\N$, for all $k\in\N$ and $\ell_1,\dots,\ell_k\in\N$ distinct, 
\[
\proba(Y_{n,\ell_1}>0,\dots,Y_{n,\ell_k}>0) \leq \prodd q1k\proba(Y_{n,\ell_q}>0).
\]
\end{Lem}
\begin{proof}
To prove the desired result, it suffices to show
\equh\label{eq:conditional}
\proba(Y_{n,\ell_k}>0\mid Y_{n,\ell_1}>0,\dots,Y_{n,\ell_{k-1}}>0) \leq \proba(Y_{n,\ell_k}>0),
\eque
for all $k\geq 2$ and $\ell_1,\dots,\ell_k\in\N$ distinct.
Observe that
\begin{multline*}
\proba(Y_{n,\ell_k}> 0\mid Y_{n,\ell_1}>0,\dots,Y_{n,\ell_{k-1}}>0) \\
= 1- \proba(Y_{n,\ell_k}=0)\frac{\proba(Y_{n,\ell_1}>0,\cdots,Y_{n,\ell_{k-1}}>0\mid Y_{n,\ell_k}=0)}{\proba(Y_{n,\ell_1}>0,\dots,Y_{n,\ell_{k-1}}>0)}.
\end{multline*}
The ratio after $\proba(Y_{n,\ell_k}=0)$ is larger than one, and this yields~\eqref{eq:conditional} and hence the desired result. To see this, let $\{Y_n^*\}_{n\in\N}$ be another collection of i.i.d.~random variables, taking values $i\in\N\setminus\{\ell_k\}$ with probability $p_i^* = p_i/(1-p_{\ell_k})$. Then, the ratio above equals
\[
\frac{\proba(Y_{n,\ell_1}^*>0,\dots,Y_{n,\ell_{k-1}}^*>0)}{\proba(Y_{n,\ell_1}>0,\dots,Y_{n,\ell_{k-1}}>0)}\geq 1.
\]
\end{proof}
Now to obtain~\eqref{eq:Knm}, it suffices to observe that
\begin{multline*}
\esp K_n^m  =  \sum_{\ell_1,\dots,\ell_m}\proba(Y_{n,\ell_1}>0,\dots,Y_{n,\ell_m}>0)\\
 \leq  \summ k1mC_{k,m}\sum_{\substack{\ell_1,\dots,\ell_k\\ \ell_i\neq\ell_j, i\neq j}}\prodd q1k \proba(Y_{n,\ell_q}>0) \leq \summ k1m C_{k,m}\Phi_n^k,
\end{multline*}
for some constants $C_{k,m}>0$,
and recall that $\Phi_n\uparrow\infty$ as $n\to\infty$. 
\end{proof}

We also need to work with partitions generated between two times $m$ and $n$, that is, the partitions generated by $Y_{m+1},\dots,Y_n$. For this purpose, we introduce $Y^*_{m,n,\ell} = \summ i{m+1}n\inddd{Y_i = \ell}$,
\eqnh
K^*_{m,n} = \sif \ell1\inddd{Y^*_{m,n,\ell}>0} \qmand K^*_{m,n,r} = \sif\ell1\inddd{Y^*_{m,n,\ell} = r}.
\eqne
We need the following lemma.
\begin{Lem}
Under the assumption~\eqref{eq:RV}, with probability one for all $s,t\in[0,1]$, $s<t$, the following limits hold:
\eqnhn
\limn\frac{K^*_{\floor{ns},\floor{nt}}}{n^\alpha L(n)} & =  &(t-s)^\alpha\Gamma(1-\alpha), \label{eq:K*}\\
\limn\frac{\sif irK^*_{\floor{ns},\floor{nt},i}}{n^\alpha L(n)} & =  &(t-s)^\alpha\Gamma(1-\alpha) \sif irp_\alpha(i),\label{eq:K*geqr}\\
\limn\frac{K^*_{\floor{ns},\floor{nt},r}}{n^\alpha L(n)} & =  &(t-s)^\alpha\Gamma(1-\alpha)p_\alpha(r), \label{eq:K*r}\\
\limn\frac{\sif i1K^*_{\floor{ns},\floor{nt},2i-1}}{n^\alpha L(n)} & =  &(t-s)^\alpha\Gamma(1-\alpha)2^{\alpha-1}. \label{eq:K*odd}
\eqnen
\end{Lem}
\begin{proof}
To prove the desired results, it suffices to establish them for fixed $s$ and $t$; the results then hold for a countable dense set of $[0,1]$ with probability one, and by continuity for all $s,t\in[0,1]$ with probability one.

Observe  that
\equh\label{eq:K*K}
K^*_{\floor{ns},\floor{nt}} \eqd K_{\floor{nt}-\floor{ns}} \qmand K^*_{\floor{ns},\floor{nt},r}\eqd K_{\floor{nt}-\floor{ns},r},\;\mbox{ for all } r\in\N,
\eque 
where `$\eqd$' stands for equality in distribution. 
By Lemma~\ref{lem:K}, it follows immediately that all convergences hold in probability. To strengthen to the almost sure sense, we apply a monotonicity argument as in \citep[Proposition 2]{gnedin07notes}.

From now on, we fix $s,t\in[0,1], s<t$. We first prove~\eqref{eq:K*}. Let $V_n = \var K_n$ and, as before, $\Phi_n = \esp K_n$. By \citep[Lemma 1 and Proposition 17]{gnedin07notes}, 
\[
V_n\sim \Gamma(1-\alpha)(2^\alpha-1)n^\alpha L(n),
\mmas n\to\infty.
\]
Therefore, for $n_m = \floor{m^{2/\alpha}}$, by~\eqref{eq:K*K} and the Borel--Cantelli lemma, we have
$$
\limn \frac{K^*_{\floor{n_ms},\floor{n_mt}}}{\Phi_{\floor{n_mt}-\floor{n_ms}} }= 1\;\text{ almost surely.}
$$
Thus, by 
\eqref{eq:Phi_n},
\[
\limm\frac{K^*_{\floor{n_ms},\floor{n_mt}}}{n_m^{\alpha} L(n_m)} = (t-s)^\alpha\Gamma(1-\alpha) \;\mbox{ almost surely.}
\]
Furthermore, for $m$ large enough, we have $\floor{n_ms}<\floor{n_{m+1}s}<\floor{n_mt}<\floor{n_{m+1}t}$, and since $\Phi_n$ is increasing, 
\[
\frac{K^*_{\floor{n_{m+1}s},\floor{n_mt}}}{\Phi_{\floor{n_{m+1}t}-\floor{n_ms}}} 
\leq \frac{K^*_{\floor{ns},\floor{nt}}}{\Phi_{\floor{nt}-\floor{ns}}}
\leq \frac{K^*_{\floor{n_{m}s},\floor{n_{m+1}t}}}{\Phi_{\floor{n_{m}t}-\floor{n_{m+1}s}}}, \mfa n_m\leq n\leq n_{m+1}.
\]
Since
\[
\limm \frac{\Phi_{\floor{n_{m+1}t}-\floor{n_ms}}}{\Phi_{\floor{n_{m}t}-\floor{n_{m+1}s}}} = 1,
\]
it follows that \eqref{eq:K*} holds with probability one. The same argument holds for \eqref{eq:K*geqr}, which implies \eqref{eq:K*r}. At last, \eqref{eq:K*odd} follows from \eqref{eq:K*geqr} and \eqref{eq:K*r}.
\end{proof}

\subsection{Proof of Theorem~\ref{thm:1}}
The main idea behind the proof is that conditioning on the underlying partitions, the partial sum $S_\vvn$ can be represented as a sum of independent $\pm1$-valued random variables, for which limit theorems follow immediately.  We illustrate this idea by first proving a central limit theorem of the model. We let $\calN(0,\sigma^2)$ denote the Gaussian distribution with mean zero and variance $\sigma^2$.
\begin{Prop}\label{prop:CLTd}
For the two-dimensional randomized Karlin model, if $\mu_{q}$ satisfies~\eqref{eq:RV} with $\alpha_q\in(0,1)$ and slowly varying functions $L_q$ for $q=1,2$, then
\[
\frac{S_\vvn}{|\vvn|^{\vv\alpha/2} \vv L(\vvn)^{1/2}}\weakto\calN(0,\sigma_{\vv\alpha}^2),
\]
as $\vvn\to\infty$, with $\sigma_{\vv\alpha}^2 = \prodd q12 \Gamma(1-\alpha_q)2^{\alpha_q-1}$. 
\end{Prop}
\begin{proof}
Let $S_{\vvn,\vvl}$ be the sum over all the spins $X_i$ associated to the box $\vvl$: 
\[
S_{\vvn,\vvl}=\sum_{\vvi\in[\vv1,\vvn]}X_\vvi\inddd{Y_\vvi=\vvl}.
\]
Because of the alternating assignment rule, $S_{\vvn,\vvl} \in\{-1,0,1\}$, and $S_{\vvn,\vvl}= 0$ if and only if the number of variables $X_\vvi$ associated to the box $\vvl$ after $\vvn$ samplings is even. 
We are therefore interested in the number of boxes having an odd number of balls after $\vvn$ samplings. To give an expression of this number, to be denoted by $\wt K_\vvn$ below, we first remark that the number of boxes with an odd number of balls from samplings $Y\topp q$ equals
\[
\sif\ell1 \inddd{Y_{n,\ell}\topp q~\rm odd} = \sif i1K\topp q_{n,2i-1},
\]
where $K\topp q_{n,i} = \sif\ell 1\inddd{Y_{n,\ell}\topp q = i}$. 
It follows that
\equh\label{eq:Knd}
\wt K_\vvn = \prodd q12\pp{\sum_{i_q=1}^\infty K\topp q_{n_q,2i_q-1}}.
\eque
Therefore,
\equh\label{eq:S_nK_n}
S_\vvn \mid \wt K_\vvn \eqd \summ i1{\wt K_\vvn}\epsilon_i',
\eque
where the left-hand side is understood as the conditional distribution of $S_\vvn$ given $\wt K_\vvn$, and on the right-hand side $\indn{\epsilon'}$ are i.i.d.~copies of $\epsilon_{\vv1}$. 
Introduce furthermore the $\sigma$-field $\calY = \sigma(Y\topp1,Y\topp 2)$, and observe that $\wt K_\vvn$ is $\calY$ measurable.
Now, by Lemma~\ref{lem:K}, 
\[
\limn\frac{\var(S_\vvn\mid\calY)}{|\vvn|^{\vv\alpha} \vv L(\vvn)} = \limn\prodd q12\frac{\sif{i_q}1K\topp q_{n_q,2i_q-1}}{n_q^{\alpha_q} L_q(n_q)} = \prodd q12\sigma_{\alpha_q}^2 = \sigma_{\vv\alpha}^2 \;\mbox{ a.s.}
\]
Therefore, we obtain the conditional central limit theorem
\[
\left.\frac{S_\vvn}{|\vvn|^{\vv\alpha/2} \vv L(\vvn)^{1/2}}\mmid\calY\right.\weakto\calN(0,\sigma_{\vv\alpha}^2), \mmas \vvn\to\infty,
\]
and the desired annealed version follows.
\end{proof}

To prove Theorem~\ref{thm:1}, we prove the convergence of finite-dimensional distributions and tightness separately.
\begin{proof}[Proof of convergence of finite-dimensional distributions]
For $m\in\N$ fixed, consider $\lambda_1,\ldots,\lambda_m\in\R$, and $\vv t\topp r = (t_{1}\topp r,t_{2}\topp r)\in[0,1]^2$ for $r=1,\ldots,m$. Writing
\equh\label{eq:nr}
\vvn_{\vv t} = (\sfloor{n_1t_{1}},\sfloor{n_2t_{2}}),\,\mbox{ for all } \vvt\in [0,1]^2,\;\vvn\in\N^2,
\eque
we set
\[
\what S_\vvn = \summ r1m\lambda_rS_{\vvn_{\vv t\topp r}},\;\vvn\in\N^2.
\]
Similarly as before and using the Cram\'er--Wold device \citep[Corollary~4.5]{kallenberg97foundations}, to show the convergence of finite-dimensional distributions, it suffices to show the following conditional central limit theorem:
\[
\left.
\frac{\what S_\vvn}{\sigma_{\vv\alpha}|\vvn|^{\vv\alpha/2} \vv L(\vvn)^{1/2}}\mmid\calY\right.\weakto \summ r1m\lambda_r\B_{\vv t\topp r}^{\vv\alpha/2},
\]
as $\vvn\to\infty$. 
For this purpose, we first remark that given $\calY$, $\what S_\vvn$ is the sum of 
$\wt K_\vvn$ independent random variables corresponding to the $\wt K_\vvn$ boxes that have at least one ball from the first $\vvn$ samplings, and that each such random variable  is bounded by $|\lambda_1|+\cdots+|\lambda_m|$ uniformly. At the same time, we know that $\wt K_\vvn\to\infty$ almost surely, as $\vvn\to\infty$. Therefore, 
to establish the conditional central limit theorem it remains to show that the variance 
\eqnh
\var(\what S_\vvn\mid\calY) = \summ r1m\summ {r'}1m\lambda_r\lambda_{r'}\esp_\calY(S_{\vvn_{\vv t\topp r}}S_{\vvn_{\vv t\topp{r'}}})
\eqne
converges to the corresponding one of the fractional Brownian sheet as $\vvn\to\infty$, after normalization. Here and in the sequel, we write $\esp_\calY(\cdot) = \esp(\cdot\mid\calY)$. This part is  established in Lemma~\ref{lem:variance}.  
\end{proof}
\begin{Lem}\label{lem:variance}
With the same notation as in~\eqref{eq:nr}, for all $\vvn\in\N^2$, $\vvt,\vvs \in [0,1]^2$,
\[
\lim_{\vvn\to\infty}\frac{\esp_\calY(S_{\vvn_{\vv t}}S_{\vvn_{\vv s}})}{\sigma_{\vv\alpha}^2|\vvn|^{\vv\alpha} \vv L(\vvn)} 
= \prodd q12\frac12\pp{|t_{q}|^{\alpha_q}+|s_{q}|^{\alpha_{q}} - |t_{q}-s_{q}|^{\alpha_q}}.
\]
\end{Lem}

\begin{proof}
We first consider the case of the one-dimensional model described in Section~\ref{sec:KarlinOne}.
We write, for $n\in\N$, $t,s\in[0,1]$,
\equh\label{eq:var2}
\esp_\calY(S_{n_t}S_{n_s}) = \frac12\esp_\calY\bb{S_{n_t}^2+S_{n_s}^2 - (S_{n_t}-S_{n_s})^2},
\eque
where $S_n=\sum_{i=1}^n X_i$, $n_t=\sfloor{nt}$, and $n_s=\sfloor{ns}$. 
We saw in the proof of Proposition~\ref{prop:CLTd} that 
$
\esp_\calY S_{n}^2 = {\sif i1 K_{n,2i-1}},
$
and thus Lemma~\ref{lem:K} yields that, almost surely,
\equh\label{eq:var3}
\limn\frac{\esp_\calY S_{n_t}^2}{n^\alpha L(n)} = t^{\alpha}\sigma_\alpha^2.
\eque
For $n>n'$, by a similar argument as in the proof of Proposition~\ref{prop:CLTd}, we see that 
\[
 S_{n}-S_{n'}\mid \calY \eqd \summ i1{\wt K^*_{n,n'}}  \epsilon_i',\quad \text{ with }\wt K^*_{n,n'}= \sif i1 K^*_{n,n',2i-1},
\]
where $\{\epsilon_i'\}$ are i.i.d.~copies of $\epsilon_1$.
In this way,
\[
\esp_\calY(S_{n_t}-S_{n_s})^2 = \sif i1 K^*_{n_t,n_s,2i-1},
\]
and by~\eqref{eq:K*odd}, 
\[
\limn\frac{\esp_\calY(S_{n_t}-S_{n_s})^2}{n^\alpha L(n)} = |t-s|^\alpha\sigma_\alpha^2\; \mbox{ almost surely}.
\]
Combining this,~\eqref{eq:var2} and~\eqref{eq:var3}, we have thus proved
\[
 \lim_{n\to\infty}\frac{\esp_\calY(S_{n_t}S_{n_s})}{\sigma_{\alpha}^2|n|^{\alpha} L(n)} 
= \frac12\pp{|t|^{\alpha}+|s|^{\alpha} - |t-s|^{\alpha}}.
\]

For the two-dimensional model, we start by introducing a different model. Let $\{\wt \epsilon\topp q_k\}_{q=1,2,\;k\in\N}$ be i.i.d.~random variables taking values $\pm1$ with equal probabilities and set $\wt\epsilon_\vvn = \prodd q12\wt\epsilon\topp q_{n_q}$. Now, assign 
\equh\label{eq:wtX}
\wt X_\vv n = \prodd q12 \wt\epsilon_{\ell_q}\topp q(-1)^{Y\topp q_{n_q,\ell_q}} \mbox{ if } Y_\vvn = \vvl,
\eque
and set $\wt S_\vvn = \sum_{\vvi\in[\vv1,\vvn]}\wt X_{\vvi}$. 
Although $\{\wt X_\vvn\}_{\vvn\in\N^2}$  is different from $\{X_\vvn\}_{\vvn\in\N^2}$, observe that for all $\vvi,\vvj\in\N^2$, 
\[
(X_\vvi,X_\vvj)\mid\calY \eqd (\wt X_{\vvi},\wt X_{\vvj})\mid\calY.
\]
Indeed, this follows from the fact that
$(\epsilon_\vvl,\epsilon_{\vvl'}) \eqd (\wt\epsilon_\vvl,\wt\epsilon_{\vvl'}), \mfa \vvl,\vvl'\in\N^2$. (Note that $\{\epsilon_\vvl\}_{\vvl\in\N^d}$ and $\{\wt\epsilon_\vvl\}_{\vvl\in\N^d}$ do not have the same joint distributions, although the fact that they have the same bivariate distributions serves our purpose.)\\
It follows that
\[
\esp_\calY(S_{\vvn_{\vv t}}S_{\vvn_{\vv s}}) = \esp_\calY(\wt S_{\vvn_{\vv t}}\wt S_{\vvn_{\vv s}}).
\]
However,
 $\esp_\calY(\wt S_{\vvn_{\vv t}}\wt S_{\vvn_{\vv s}})$ is much easier to compute. From~\eqref{eq:wtX}, $\wt X_\vvi$ can be written as
\[
\wt X_\vvi = \prodd q12 \wt X\topp q_{i_q},\, \mbox{ if }Y_\vvi = \vvl,\qmwith \wt X\topp q_{i_q} = \wt\epsilon\topp q_{\ell_q}(-1)^{Y\topp q_{i_q,\ell_q}+1},\; q=1,2.
\]
In this way, one can write
\[
\wt S_\vvn = \sum_{\vvi\in[\vv1,\vvn]}\prodd q12\wt X\topp q_{i_q} = \prodd q12\sum_{i_q=1}^{n_q}\wt X\topp q_{i_q} = \prodd q12 \wt S\topp q_{n_q} \qmwith \wt S\topp q_n = \summ i1{n}\wt X\topp q_i.
\]
Observe that $\{\wt X\topp1_i\}_{i\in\N}$ and $\{\wt X\topp2_i\}_{i\in\N}$ are independent and each $\wt S_n\topp q$ is the partial sum of a one-dimensional Karlin model with parameter $\alpha_q$.  Therefore, 
\begin{align*}
\frac{\esp_\calY(\wt S_{\vvn_{\vvt}}\wt S_{\vvn_{\vvs}})}{|\vvn|^{\vv\alpha} \vv L(\vvn)} 
&= \prodd q12\frac{\esp_\calY\pp{\wt S\topp q_{\floor{n_q t_q}}\wt S\topp q_{\floor{n_qs_q }}}}{n_q^{\alpha_q}L_q(n_q)}\\
&\to \sigma_{\vv\alpha}^2\prodd q1d\frac12\pp{|t_{q}|^{\alpha_q} + |s_{q}|^{\alpha_{q}} - |t_{q}-s_{q}|^{\alpha_q}},\mmas \vvn\to\infty.
\end{align*}
\end{proof}

\begin{proof}[Proof of tightness]
Applying a criterion of \citet{bickel71convergence},
it suffices to establish for some $p>0, \gamma>1$, 
\equh\label{eq:bickel}
\esp\pp{\frac{|S_\vvm|}{|\vvn|^{\vv\alpha/2}}}^p\leq C\prodd q12\pp{\frac{m_q}{n_q}}^{\gamma}, \mfa \vvm,\vvn\in\N^2, \vvm\leq\vvn.
\eque
To do so, pick $p>\max(2/\alpha_1,2/\alpha_2)$.
Recall that, given $\calY$, $S_\vvm$ is the sum of $\wt K_\vvm$ independent copies of 
$\epsilon_{\vv1}$.
We infer
\[
\esp|S_\vvm|^p = \esp\esp_\calY(|S_\vvm|^p) 
\leq \esp \bb{C_{p}\esp_\calY{\wt K_\vvm^{p/2}}}
 = C_p\esp \wt K_\vvm^{p/2},
\]
where we used Burkholder's inequality, and $C_p$ is a positive constant depending only on $p$. The expectation on the right-hand side above is then bounded from above by, recalling~\eqref{eq:Knd},
\[
\esp\wt K_\vvm^{p/2}
= \prodd q12\esp\pp{\sif{i_q}1K\topp q_{m_q,2i_q-1}}^{p/2} \leq \prodd q12\esp(K\topp q_{m_q})^{p/2}.
\]
Now, for each $q$, the expectation can be uniformly bounded by $C_q m_q^{\alpha_qp/2}L_q(m_q)^{p/2}$ for some constant $C_q>0$ by Lemma~\ref{lem:K}. Therefore,
\[
\esp\pp{\frac{|S_\vvm|}{|\vvn|^{\vv\alpha/2}}}^{p}\leq C_p \prodd q12 C_q\pp{\frac{m_q}{n_q}}^{\gamma'}\left(\frac{L_q(m_q)}{L_q(n_q)}\right)^{p/2}
\]
with $\gamma' = \min(\alpha_1,\alpha_2)p/2 > 1$. To conclude, we choose $\delta>0$ such that $\gamma=\gamma'-\delta p/2 >1$ and we apply Potter's Theorem (see \cite[Theorem 1.5.6]{bingham87regular}) to bound from above 
${L_q(m_q)}/{L_q(n_q)}$ by $C({m_q}/{n_q})^{-\delta}$. The inequality \eqref{eq:bickel} follows
and we have thus proved the tightness.
\end{proof}


\section{Hammond--Sheffield model}\label{sec:HS}
In this section, we introduce the two-dimensional Hammond--Sheffield model and show that the partial-sum random field scales to a fractional Brownian sheet with  Hurst index in $(1/2,1)^2$. 

\subsection{One-dimensional model}
We first recall the model in one dimension. 
Let $\mu$ be a probability measure on $\N$ satisfying
\[
\mu(\{n,n+1,\dots\}) \sim n^{-\alpha}L(n)
\]
with $\alpha\in(0,1/2)$ and $L$ a slowly varying function at infinity. Let $\{J_i\}_{i\in\Z}$ be i.i.d.~random variables with distribution $\mu$ and consider the random graph 
$G=G(V,E)$ with vertex set $V = \Z$ and edge set $E = \{(i,i-J_i)\}_{i\in\Z}$.
In words, for each vertex $i\in\Z$, a random jump $J_i$ is sampled independently from $\mu$ and the vertex $i$ is connected to the vertex $i-J_i$. 
For each vertex $i$, the largest connected subgraph of $G$ containing $i$ is a tree with an infinite number of vertices. 
Each such tree is referred to as  a component of $\Z$. It was shown in \citep{hammond13power} that for $\alpha\in(0,1/2)$, the random graph $G$ almost surely has infinitely many components, each being an infinite tree. 
The random forest $G$ obtained this way induces a random partition of $\Z$, so that $i$ and $j$ are in the same component, denoted by $i\sim j$,  if and only if they are in the same tree. In the sequel, it is  convenient to work with {\it ancestral lines} defined as the random sets
\[
A_i = \{j\in\Z:\exists\, j=j_0<j_1<\cdots<j_k=i,
 \mbox{ s.t. } (j_{\ell-1},j_\ell)\in E,\, \ell=1,\dots,k\}\cup\{i\},
\]
for all $i\in\Z$. So, $i\sim j$ if and only if $A_i\cap A_j\neq\emptyset$.

We now apply the identical assignment rule. This entails that  marginally $\proba(X_i = -1) = \proba(X_i =1) =1/2$, and conditioning on $\calG = \sigma\{J_i,\,i\in\Z\}$, $X_i=X_j$ if $A_i\cap A_j\neq\emptyset$, and $X_{i_1},\dots,X_{i_k}$ are independent for any $i_1,\dots,i_k$ such that $A_{i_1},\dots,A_{i_k}$ are mutually disjoint. The one-dimensional Hammond--Sheffield model is the stationary process $\{X_i\}_{i\in\Z}$ constructed this way.

The following notations and results from \citep{hammond13power}  will be used in our two-dimensional model. 
Let, for $k\in\Z$,
\eqnh
q_k = \proba(0\in A_k),
\eqne
so $q_k = 0$ for $k<0$. 
It is proved in \citep[Lemma~3.1]{hammond13power} that with the choice of $\mu$ above and $\alpha\in(0,1/2)$,
$\sif k0 q_k^2<\infty$, and for $S_n=\sum_{i=1}^n X_i$, 
\equh\label{eq:VarSn}
\var(S_n)
\sim \frac{C_{\alpha}}{\sum_{k\ge0}q_k^2} n^{2\alpha+1}L(n)^{-2}, \mmas n\to\infty,
\eque
with
\equh\label{eq:Calpha}
C_\alpha = \frac{\sin(\pi\alpha)}{\pi \alpha(2\alpha+1)\Gamma(1-2\alpha)}.
\eque

We shall however need a slightly more general version when working with the two-dimensional model later.  
We say that $\{X_i\}_{i\in\Z}$ is a {\em generalized one-dimensional Hammond--Sheffield model} with distribution $\nu$ on $\R$ if it is built using $\nu$ as the common marginal distribution instead of the symmetric law on $\{-1,1\}$. That is, the underlying random partition is the same as before and, conditioning on the random partition, the identical assignment rule is applied ($X_i = X_j$ if $i\sim j$ and $X_i$ and $X_j$ are independent otherwise) with each $X_i$ having the same marginal distribution $\nu$.

\begin{Prop}\label{prop:generalizedHS}
For the generalized one-dimensional Hammond-Sheffield model with a centered distribution $\nu$ with bounded support, 
\[
\ccbb{\frac{S_{\floor{nt}}}{n^{H}L(n)\inv}}_{t\in[0,1]}\weakto \pp{\frac {C_\alpha}{\sum_{k\ge 0}q_k^2}}^{1/2}\ccbb{\B^H_t}_{t\in[0,1]}
\]
in $D([0,1])$ with $H = \alpha+1/2$.
\end{Prop}
\begin{Rem}\label{rem:boundedHS} 
The results in \citep{hammond13power} concern only $\nu$ supported on $\{-1,1\}$. 
The relaxation of $\nu$ to bounded law does not affect most of the proof, which is based on a martingale central limit theorem. The boundedness is
sufficient for certain ergodicity of the sequence of martingale differences (\citep[Lemma 3.2]{hammond13power} and \citep[Lemma 7]{bierme17invariance}), and the rest of the proof would remain unchanged. 
As the proof for the two-dimensional model will follow the same strategy but is much more involved, we therefore skip the proof of Proposition  \ref{prop:generalizedHS} here. 
\end{Rem}

\subsection{Two-dimensional model and main result}
We now generalize Hammond--Sheffield model to two dimensions. Again the first step is to construct a random partition of $\Z^2$. This random partition is taken as the product of independent random partitions from one-dimensional Hammond--Sheffield models, each with jump distribution $\mu_{r}$, $r=1,2$ respectively, satisfying
\equh\label{eq:RVHS}
\mu_{r}(\{n,n+1,\dots\}) \sim  n^{-\alpha_r}L_r(n), \mmas n\to\infty
\eque
with $\alpha_r\in(0,1/2)$ and slowly varying function $L_r$ at infinity. For $r=1,2$, let $\{A_i\topp r\}_{i\in\Z}$ be the ancestral lines corresponding to each random partition. In particular, $\{A_i\topp 1\}_{i\in\Z}$ and $\{A_i\topp 2\}_{i\in\Z}$ are independent. We then introduce the {\it ancestral lattices} $A_\vvi, \vvi\in\Z^2$, as
\[
A_\vvi = \ccbb{\vvj\in\Z^2\mid j_1\in A_{i_1}\topp 1,j_2\in A_{i_2}\topp 2} = A_{i_1}\topp 1\times A_{i_2}\topp2.
\]
For the partition of $\Z^2$ obtained by product, we have $\vvi\sim \vv j$ if and only if $A_\vvi\cap A_\vvj \neq\emptyset$.
Once the random partition is given, the identical assignment rule is applied. That is, given $\{A_\vvi\}_{\vvi\in\Z^2}$, $X_\vvi = X_\vvj$ if $A_\vvi\cap A_\vvj\neq\emptyset$, and if $A_{\vvi_1},\dots,A_{\vvi_k}$ are mutually disjoint, $X_{\vvi_1},\cdots,X_{\vvi_k}$ are i.i.d.~with common distribution the uniform law on $\{-1,1\}$. 
The so-constructed $\{X_\vvi\}_{\vvi\in\Z^2}$ is referred to as the two-dimensional Hammond--Sheffield model in the sequel. 

We write for $\vvn\in\Z^2$, $q_{\vvn} = \proba(\vv0\in A_\vvn)$, and $q_n\topp r = \proba(0\in A_n\topp r)$, $r=1,2$. Because of independence, 
\eqnh\
q_\vvn = q_{n_1}\topp 1\, q_{n_2}\topp 2,
\eqne
and then $\sum_{\vvn\in\Z^2}q_\vvn^2< \infty$ when  $(\alpha_1,\alpha_2)\in(0,1/2)^2$.
The main result of this section is the following functional central limit theorem, where as before $\vvn^{\vv\alpha}=\prod_{r=1}^2 n_r^{\alpha_r}$ and $\vv L(\vv n)=\prod_{r=1}^2 L_r(n_r)$.

\begin{Thm}\label{thm:HSproduct}
For the two-dimensional Hammond--Sheffield  model, suppose~\eqref{eq:RVHS} holds with $\alpha_1,\alpha_2\in(0,1/2)$ and slowly varying functions $L_1$, $L_2$ respectively. For $S_\vvn = \sum_{\vvi\in[\vv1,\vvn]}X_\vvi$, we have
\[
\ccbb{\frac{S_{\floor{\vvn\cdot\vv t}}}{|\vvn|^{\vvH}\vv L(\vvn)^{-1}}}_{\vvt\in[0,1]^2}\weakto \sigma_{\vv\alpha} \ccbb{\B_\vvt^{\vv H}}_{\vvt\in[0,1]^2}
\]
in $D([0,1]^2)$, as $\vvn\to \infty$, where $\B^{\vvH}$ is a fractional Brownian sheet with Hurst index $\vvH$ with $H_r = \alpha_r+1/2$, $r=1,2$, and 
\eqnh
\sigma_{\vv\alpha}^2 = \frac{C_{\alpha_1} C_{\alpha_2}}{\sum_{\vvn\in\Z^2}q_\vvn^2}.
\eqne
with $C_\alpha$ defined in \eqref{eq:Calpha}.
\end{Thm}

\begin{Rem}\label{rem:HS}
Another natural extension of the Hammond--Sheffield model has been addressed in \citep{bierme17invariance}, where the random graph indexed by $\Z$ in the original model is generalized to high dimensions by having i.i.d.~jumps attached to vertices indexed by $\Zd$ and allowing each jump to take values in $\N^d$. With appropriate choice of the law of the jumps, the limit random fields therein are of different types from  fractional Brownian sheets most of the time (even when fractional Brownian sheets arise in the limit, they are degenerate in the sense that at least one of the Hurst indices is either $1/2$ or $1$ \citep[Section 5.2]{bierme17invariance}), and a so-called {\em scaling-transition} phenomenon \citep{puplinskaite16aggregation,puplinskaite15scaling} occurs. 
In particular, the partial sum of interest therein is still over a rectangular region that increases to infinity, although for the same model (i.e.~with fixed law of the jumps) various limits arise, depending on the relative growing rate of each direction of the increasing rectangle. 
\end{Rem}
The rest of this section is devoted to the proof of Theorem \ref{thm:HSproduct}. 
The strategy is to express the partial sum of the variable $X_\vvi$ as a weighted sum of martingale differences in the first direction and to apply a theorem of \citet{mcleish74dependent} for triangular arrays of martingale differences. The hard part lies in the analysis of the second direction, where we shall apply results for the generalized one-dimensional Hammond--Sheffield model.

\subsection{Representation via martingale differences}\label{sec:representation}
Introduce for each $m\in\Z$, the $\sigma$-algebra of the past in the first direction 
$\calF_{m}\topp1 = \sigma\{X_\vvi: i_1<m, i_2\in\Z\}$
and the operators
\[
\calP_m\topp1(\cdot) = \esp(\cdot\mid\calF_{m+1}\topp1) - \esp(\cdot\mid\calF_{m}\topp1),\quad  m\in\Z.
\]
Observe that $\calP_m\topp1(Y) \in \calF_{m+1}\topp1$ and $\esp(\calP_m\topp1(Y)\mid\calF_{m}\topp1) = 0$ for any bounded random variable $Y$.
Introduce 
\equh\label{eq:X*}
 X_\vvj^* = \calP_{j_1}\topp1(X_\vvj) = X_\vvj-\esp(X_\vvj \mid \calF_{j_1}\topp1),\quad \vvj\in\Zt.
 \eque
By definition, for all $j_2\in\Z$, $\{X_\vvj^*\}_{j_1\in\Z}$ is a martingale-difference sequence with respect to the filtration $\{\calF_{j_1}\}_{j_1\in\Z}$.
Denoting by $\{J_j\topp1\}_{j\in\Z}$ the random jumps in the first direction and observing that for all $\vvj\in\Z^2$, 
\[
X_\vvj=\sum_{k\ge 1} \inddd{J_{j_1}\topp1 = k} X_{(j_1-k,j_2)},
\] we obtain another representation of $X_\vvj^*$ as
\equh\label{eq:HSHS_Xj}
X_\vvj^* = X_\vvj - \sum_{k\ge 1}p_{k}\topp1X_{(j_1-k,j_2)},\quad \vvj\in\Z^2,
\eque
where $p_{k}\topp1=\mu_1(\{k\})$, $k\in\N$.
Recall that $q_\vvn =0$ if $\min(n_1,n_2)<0$. We have the following results.
\begin{Lem}\label{lem:projection}
\noindent (i) For all $m\in\Z$, $\vvn\in\Zt$,  $\calP_m\topp1(X_\vvn) = q_{n_1-m}\topp1 X_{(m,n_2)}^*.$

\noindent(ii) For all $\vvn\in\Zt$,
\eqnh
X_\vv n = \sum_{m\leq n_1}q_{n_1-m}\topp1 X_{(m,n_2)}^*,
\eqne
where the sum converges in $L^2$. Furthermore
$
\var(X_{\vv0}^*) = (\sum_{k\geq 0}(q_k\topp1)^2)^{-1}.
$
\end{Lem}

\begin{proof}
(i) For $m\in\Z$, write
\[
X_\vvn = X_\vvn\inddd{m\in A_{n_1}\topp1} + X_\vvn\inddd{m\not\in A_{n_1}\topp1}.
\]
Observe that 
\[
X_\vvn\inddd{m\in A_{n_1}\topp1} = X_{(m,n_2)}\inddd{m\in A_{n_1}\topp1}
\] and $\{m\in A_{n_1}\topp1\}$ is independent of $\calF_{m+1}\topp1$. It then follows that
\[
\calP_m\topp1\pp{X_\vvn\inddd{m\in A_{n_1}\topp1}} = \calP_m\topp1\pp{X_{(m,n_2)}}\proba\pp{m\in A_{n_1}\topp1} =  X_{(m,n_2)}^*q_{n_1-m}\topp1.
\] 
On the other hand, 
\[
\esp\pp
{X_\vvn\inddd{m\not\in A_{n_1}\topp1}\mmid \calF_{m+1}\topp1}=\esp\pp{X_\vvn\inddd{m\not\in A_{n_1}\topp1}\mmid \calF_{m}\topp1},
\] and thus 
\[
\calP_m\topp1\pp{X_\vvn\inddd{m\not\in A_{n_1}\topp1}}=0.
\]

(ii) By stationarity, it suffices to prove this for $X_\vv0$. 
For $n\in\N$, write
\[
\sum_{m=-n}^0\calP_m\topp1(X_\vv0) = \sum_{m=-n}^0 q_{m}\topp1 X_{(m,0)}^*.
\]
Since $\{X_{(m,0)}^*\}_{m\in\Z}$ is a stationary martingale-difference sequence, we have  that $\esp X_{(m,0)}^* = 0$ and $\esp(X_{(m,0)}^* X_{(n,0)}^*) = 0$ if $m\neq n$. 
Then,
\equh\label{eq:tail1}
\esp\pp{\sum_{m=-n}^0\calP_m\topp1 (X_\vv0)}^2 = \var(X_\vv0^*) \sum_{m=0}^n (q_m\topp1)^2 \to \var(X_\vv0^*)\sum_{m\geq0} {(q_m\topp1)^2},\mmas n\to\infty.
\eque
Here the assumption $\alpha_1\in(0,1/2)$ entails that $\sum_{m\geq0}{(q_m\topp1)^2}<\infty$.

On the other hand, let $J_0\topp1\in\N$ denote the random jump at $0$ in the first direction. 
One can write, in view of~\eqref{eq:HSHS_Xj},
\[
X_{\vv0}^* = \sum_{\ell\in\N}\pp{\inddd{J_0\topp1=\ell}-p_{\ell}\topp1} X_{(-\ell,0)}.
\]
Thus,
\begin{align*}
\var(X_{\vv0}^*) 
& = \sum_{\ell,\ell'\in\N}\proba\pp{A_{-\ell}\topp1\cap A_{-\ell'}\topp1\neq\emptyset} 
\esp \bb{\pp{\inddd{J_0\topp1=\ell}-p_{\ell}\topp1}\pp{\inddd{J_0\topp1=\ell'}-p_{\ell'}\topp1}}\\
&= \sum_{\ell,\ell'\in\N}\proba\pp{A_{-\ell}\topp1\cap A_{-\ell'}\topp1\neq\emptyset} \pp{-p_{\ell}\topp1 p_{\ell'}\topp1 + p_{\ell}\topp1 \inddd{\ell = \ell'}}\\
&=1 -\sum_{\ell,\ell'\in\N} p_{\ell}\topp1 p_{\ell'}\topp1\proba\pp{A_{-\ell}\topp1\cap A_{-\ell'}\topp1\neq\emptyset} = 1-\proba\pp{A_0\topp 1\cap \wt A_0\topp 1\neq \{0\}},
\end{align*}
where  $\wt A_0\topp1$ is an independent copy of $A_0\topp1$.
Therefore,
\begin{align*}
\var(X_{\vv0}^*)  = \proba\pp{A_0\topp1\cap \wt A_0\topp1 = \{0\}}.
\end{align*}
Combining with \eqref{eq:tail1}, we have
\[
\lim_{n\to\infty} \esp\pp{\sum_{m=-n}^0\calP_m\topp1 (X_\vv0)}^2 = \proba\pp{A_0\topp1\cap \wt A_0\topp1 = \{0\}} \sum_{m\geq 0}{(q_m\topp1)^2}.
\]
We thus have the convergence in $L^2$ by the fact that $\esp X_{\vv0}^2 = 1$ and 
\equh\label{eq:geometric}
\proba\pp{A_{0}\topp1\cap \wt A_{0}\topp1 = \{0\}}=\frac1{\sum_{m \geq0}{(q_m\topp1)^2}}.
\eque
Indeed, observe that
\[
\sum_{m\geq0}{(q_m\topp1)^2}  = \sum_{m\ge 0} \proba\pp{-m\in A_0\topp1,-m\in \wt A_0\topp1} =\esp\abs{A_0\topp1\cap \wt A_0\topp1},
\]
and remark that $|A_0\topp 1\cap \wt A_0\topp1|$, the cardinality of intersection of the two independent ancestral lines, is a geometric random variable with rate $\theta = \proba(A_0\topp1 \cap\wt A_0\topp1 = \{0\})$. Thus $\esp |A_0\topp1\cap\wt A_0\topp1| = 1/\theta$, which proves \eqref{eq:geometric}.
\end{proof}
Introduce $b_{n,j}\topp1= \sum_{i=1}^nq_{i-j}\topp1$, $n\in\N$, $j\in\Z$. From the preceding lemma, we have for all $\vvn\in\N^2$,
\[
S_\vvn = \sum_{\vvi\in[\vv1,\vvn]}X_\vvi= \sum_{j_1\in\Z}b_{n_1,j_1}\topp1 
\sum_{j_2=1}^{n_2} X_{\vvj}^*.
\]
Further, for each $n\in\N$, the sequence $\spp{\sum_{j_2=1}^n X_{\vvj}^*}_{j_1\in\Z}$ is a martingale-difference sequence with respect to the filtration $\{\calF_{j_1}\topp1\}_{j_1\in\Z}$.
Denoting $(b_n\topp1)^2=\sum_{j\in\Z}(b_{n,j}\topp1)^2$, by \eqref{eq:VarSn}, we obtain
\begin{equation}\label{eq:asymb_n}
 (b_n\topp1)^2 \sim C_{\alpha_1} n^{2\alpha_1 +1}L_1(n)^{-2},\mmas n\to\infty.
\end{equation}
Now introduce similarly $b_{n,j}\topp2= \sum_{i=1}^nq_{i-j}\topp2$ and $(b_n\topp2)^2=\sum_{j\in\Z}(b_{n,j}\topp2)^2$, for $n\in\N$, $j\in\Z$. In summary, we have shown the following.
\begin{Prop}\label{prop:representation}
In the notation above, 
\equh\label{eq:Sn/bn}
\frac{S_\vvn}{b_\vvn} = \frac1{b_{n_1}\topp1}\sum_{j_1\in\Z}b_{n_1,j_1}\topp1 U_{j_1,n_2}
\eque
with
\equh\label{eq:Uj_1,n_2}
U_{j_1,n_2} = \frac1{b_{n_2}\topp2}\sum_{j_2=1}^{n_2}X_{\vvj}^*
\eque
and
\equh\label{eq:asympb_n}
b_\vvn^2=(b_{n_1}\topp1)^2 (b_{n_2}\topp2)^2 \sim C_{\alpha_1 } C_{\alpha_2} \vvn^{2\vv\alpha +\vv1} \vv L(\vvn)^{-2},\mmas \vvn\to\infty.
\eque
\end{Prop}
Note that again, for each $n\in\N$, $\{U_{j,n}\}_{j\in\Z}$ is a stationary martingale-difference sequence with respect to the filtration $\{\calF_{j}\topp1\}_{j\in\Z}$. 

\subsection{A central limit theorem}\label{sec:HSHS_CLT}
Instead of proving directly the convergence of finite-dimensional distributions, we prove the following central limit theorem first, in order to better illustrate the key ideas of the proof.
\begin{Prop}\label{prop:HSHS_CLT}
For the two-dimensional Hammond--Sheffield model, suppose \eqref{eq:RVHS} holds with $\alpha_1,\alpha_2\in(0,1/2)$
and slowly varying functions $L_1$, $L_2$ respectively.
We have
\[
\frac{S_\vvn}{b_\vvn}\weakto\calN(0,\sigma^2), \mmas \vvn\to\infty, 
\]
where $\sigma^2=\spp{\sum_{\vvk\ge \vv0}q_\vvk^2}^{-1}$. 
\end{Prop}
The rest of this subsection is devoted to the proof of this proposition. 
With the representation in \eqref{eq:Sn/bn}, by McLeish's martingale central limit theorem \citep{mcleish74dependent}, it suffices to show
\equh\label{eq:McLeish1}
\sup_{\vvn\in\N^2}\esp\pp{\sup_{j\in\Z}\pp{\frac{b_{n_1,j}\topp1}{b_{n_1}\topp1}}^{2}U_{j,n_2}^2}<\infty,
\eque
\equh\label{eq:McLeish2}
\lim_{\vvn\to\infty}\sup_{j\in\Z}\pp{\frac{b_{n_1,j}\topp1}{b_{n_1}\topp1}}^{2}U_{j,n_2}^2 = 0\quad \mbox{ in probability,}
\eque
and
\equh\label{eq:McLeish3}
\lim_{\vvn\to\infty}\sum_{j\in\Z}\pp{\frac{{b_{n_1,j}\topp1}}{{b_{n_1}\topp1}}}^2U_{j,n_2}^2 =  \esp X_{\vv0}^{*2} \quad \mbox{ in probability.}
\eque

We start with the following  observation.
\begin{Lem}\label{lem:sup_bnj}
For $\alpha_1\in(0,1/2)$, we have
\[
\limn \sup_{j\in\Z}\frac{b_{n,j}\topp1}{b_n\topp1} = 0. 
\]
\end{Lem}
\begin{proof}
By Lemma~8 in \citep{bierme17invariance}, it suffices to prove that $\sum_{j\in\Z}\pp{ (b_{n,j}\topp1)^2-(b_{n,j+1}\topp1)^2} =o((b_n\topp1)^2)$, which follows from 
$\sum_{j\in\Z}\pp{ {b_{n,j}\topp1}-{b_{n,j+1}\topp1} }^2=o((b_n\topp1)^2)$ by the Cauchy--Schwarz inequality. To see the latter, observe that 
$$\sum_{j\in\Z}\pp{ {b_{n,j}\topp1}-{b_{n,j+1}\topp1} }^2=\sum_{j\in\Z} \pp{q_{n-j}\topp1-q_{-j}\topp1}^2 \le 2 \sum_{j\in\Z} {(q_j\topp1)^2} <\infty.$$
\end{proof}

We also need uniform bounds on the moments of $U_{j,n}$. To facilitate we  introduce a representation of $U_{j,n}$ as a weighted sum of martingale differences in the second direction.
Let $\calF\topp2_m=\sigma\{X_\vvi\mid i_1\in\Z, i_2<m\}$ and $\calP_m\topp2(\cdot)=\esp(\cdot\mid \calF_{m+1}\topp2)-\esp(\cdot \mid \calF_m\topp2)$, $m\in\Z$.
We set 
\equh\label{eq:defX**}
X_\vvn^{**}=\calP_{n_2}\topp2(X_\vvn^*)=X_\vvn^*-\esp(X_\vvn^*\mid \calF_{n_2}\topp2).
\eque
For all $n_1\in\Z$, $(X_\vvn^{**})_{n_2\in\Z}$ is a martingale-difference sequence with respect to the filtration $(\calF_{n_2}\topp2)_{n_2\in\Z}$.
Proceeding as in Lemma~\ref{lem:projection}, we obtain that for all $\vvn\in\N^2$,
\[
 X_n^*=\sum_{m\le n_2} \calP_{m}\topp2(X_n^*)=\sum_{m\le n_2} q_{n_2-m}\topp2 X_{(n_1,m)}^{**},
\]
where the sum converges in $L^2$. 
We thus have, for all $n\in\N$, $j_1\in\Z$,
\equh\label{eq:U_n}
 U_{j_1,n}=\frac{1}{b_{n}\topp2}\sum_{j_2\in\Z} b_{n,j_2}\topp2 X_\vvj^{**}.
\eque
Further, 
\[
 \var(X_\vv0^{**})= \frac{\var(X_\vv0^*)}{\sum_{k\ge 0} {(q_k\topp2)^2}}=\frac1{\sum_{\vvk\ge \vv0} q_\vvk^2}.
\]

\begin{Lem}\label{lem:Umoment}
\noindent (i) For all $n\in\N$, $\esp U_{0,n}^2= \pp{\sum_{\vvk\in\Z^2}q_{\vvk}^2}^{-1}<\infty$.

\noindent (ii) For all $p\ge 1$, $\sup_{n\in\N} \esp U_{0,n}^{2p} <\infty$.
\end{Lem}

\begin{proof}
Part (i) is a direct consequence of \eqref{eq:defX**}:
we have $\esp (X_\vvn^{**}X_\vvm^{**})=0$ for $\vvn\ne\vvm$ and thus
\[
 \var(U_{0,n})= \var(X_\vv0^{**})=\frac1{\sum_{\vvk\ge 0} q_\vvk^2}.
\]
For (ii), using that $(X_{(0,n)}^{**})_{n\in\Z}$ is a  martingale-difference sequence, by Burkholder's inequality,  writing $\left\|\,\cdot\, \right\|_{p} = (\esp|\cdot|^{p})^{1/p}$, for some constant $C_p$ depending only on $p$, 
\begin{align*}
\left\| U_{0,n}\right\|_{2p} & \le C_p \left\|\sum_{j\in\Z}\frac{(b_{n,j}\topp2)^2}{(b_n\topp2)^2}X_{(0,j)}^{**2}\right\|_p^{1/2} 
\\
&  \le C_p \pp{\sum_{j\in\Z}\frac{(b_{n,j}\topp2)^2}{(b_n\topp2)^2}\left\|X_{\vv0}^{**2}\right\|_p}^{1/2} = C_p\left\|X_{\vv0}^{**}\right\|_{2p}.
\end{align*}
Then (ii) follows since $X_\vv0^{**}$ is bounded.
\end{proof}

Now, we establish the conditions of McLeish's theorem.
\medskip

For~\eqref{eq:McLeish1}, by  the inequality $\sup_j|a_j| \le \sum_j|a_j|$, the left-hand side is bounded by $\esp U_{0,n}^2 = \spp{\sum_{\vvk\ge \vv0}q_\vvk^2}^{-1}<\infty$ by Lemma~\ref{lem:Umoment} (i).
\medskip

For~\eqref{eq:McLeish2}, for all $\varepsilon>0$, one has
\begin{align*}
\proba\pp{\max_{j\in\Z}\pp{\frac{b_{n_1,j}\topp1}{b_{n_1}\topp1}}^{2}U_{j,n_2}^2>\varepsilon} & \le \sum_{j\in\Z}\proba\pp{\pp{\frac{b_{n_1,j}\topp1}{b_{n_1}\topp1}}^{2}U_{j,n_2}^2>\varepsilon}\\
& \le \sum _{j\in\Z}\pp{\frac{b_{n_1,j}\topp1}{b_{n_1}\topp1}}^{4}\frac{\esp |U_{0,n_2}|^{4}}{\varepsilon^{2}}.
\end{align*}
Lemma~\ref{lem:sup_bnj} and Lemma~\ref{lem:Umoment} (ii) then yield \eqref{eq:McLeish2}.
\medskip

Condition~\eqref{eq:McLeish3} is much harder to establish. We shall prove the corresponding  $L^2$-convergence, which will follow from 
\equh\label{eq:ergodicity}
\lim_{\vvn\to\infty}\frac1{(b_{n_1}\topp1)^4}\sum_{j_1,j_1'\in\Z}(b_{n_1,j_1}\topp1)^2(b_{n_1,j_1'}\topp1)^2\cov\pp{U_{j_1,n_2}^2,U_{j_1',n_2}^2}=0.
\eque
For this purpose, we first provide an approximation of $X_\vvj^*$ as follows.
Introduce, for each integer $K\ge 1$, for all $\vvj\in\Z^2$,
\eqnh
X_{\vvj,K}^* = X_\vvj - \sum_{k_1=1}^Kp_{k}\topp1X_{(j_1-k,j_2)}.
\eqne
Recalling \eqref{eq:HSHS_Xj}, observe that
\[
|X_{\vv0}^*-X_{\vv0,K}^*|\le \sum_{k=K+1}^\infty p_k\topp1\to 0, \mmas K\to\infty.
\]
Then, define 
\equh\label{eq:UjnK}
U_{j_1,n,K} = \frac1{b_n\topp 2}\sum_{j_2=1}^{n} X_{\vvj,K}^*,\;\text{ for } n,K\in\N,j_1\in\Z.
\eque
Note that $\{U_{j,n,K}\}_{j\in\Z}$ for every $K,n\in\N$ is again a stationary martingale-difference sequence with respect to the filtration $\{\calF_j\topp1\}_{j\in\Z}$. We shall need the following uniform bounds. 
\begin{Lem}\label{lem:Uapprox}
\noindent (i) For all $p\ge 1$ and $K\ge 1$, 
$$ \sup_{n\in\N} \esp |U_{0,n,K}|^{2p} <\infty.$$

\noindent (ii) For all $p\ge 1$,
$$ \lim_{K\to\infty}\sup_{n\in\N}\esp|U_{0,n}-U_{0,n,K}|^{2p} =0. $$
\end{Lem}
\begin{proof}
This lemma can be established in the same way as for Lemma~\ref{lem:Umoment} before by proving that for all $p\ge 1$, there exists a finite constant $C_p$ depending on $p$ only, such that for all $n,K\in\N$, 
\begin{align*}
 \esp |U_{0,n,K}|^{2p} & \le C_p\esp |X_{\vv0,K}^{*}|^{2p}\\
  \esp|U_{0,n}-U_{0,n,K}|^{2p} & \le C_p\esp|X_{\vv0}^*-X_{\vv0,K}^*|^{2p}. 
\end{align*} 
\end{proof}

Now, to prove \eqref{eq:ergodicity}, we first show that for all $\varepsilon>0$, one can choose $K\in\N$ large enough such that
\equh\label{eq:U_K}
\abs{\cov\pp{U_{0,n}^2,U_{j,n}^2} - \cov\pp{U_{0,n,K}^2,U_{j,n,K}^2}}<\varepsilon, \mfa n,j\in\N.
\eque
To see this, we first bound
\begin{multline*}
\abs{\esp(U_{0,n}^2U_{j,n}^2) - \esp(U_{0,n,K}^2U_{j,n,K}^2)} \\ 
\le \esp\abs{U_{0,n}^2(U_{j,n}^2-U_{j,n,K}^2)} + \esp\abs{(U_{0,n}^2-U_{0,n,K}^2)U_{j,n,K}^2}.
\end{multline*}
The first term on the right-hand side is bounded, applying the Cauchy--Schwarz inequality twice, by
\[
\pp{\esp |U_{0,n}|^4}^{1/2}\pp{\esp |U_{j,n}+U_{j,n,K}|^4}^{1/4} \pp{\esp |U_{j,n}-U_{j,n,K}|^4}^{1/4}.
\]
By Lemma~\ref{lem:Uapprox},  this expression  converges to $0$ uniformly in $n$, as $K\to\infty$. The second term can be treated similarly.
Therefore~\eqref{eq:U_K} follows for $K$ large enough and hence to show~\eqref{eq:ergodicity} it suffices to establish, for $K$ large enough,
\equh\label{eq:ergodicity_K}
\lim_{\vvn\to\infty}\frac1{(b_{n_1}\topp1)^4}\sum_{j_1,j_1'\in\Z}(b_{n_1,j_1}\topp1)^2 (b_{n_1,j_1'}\topp1)^2\cov\pp{U_{j_1,n_2,K}^2,U_{j_1',n_2,K}^2}=0.
\eque
For this purpose, we shall establish the following lemma.
\begin{Lem}\label{lem:cov_decay}
For all $K\in\N,\varepsilon>0$, there exist integers $L_{K,\varepsilon}, N_{K,\varepsilon}$, such that
\[
\abs{
\cov(U_{0,n,K}^2,U_{j,n,K}^2)}<\varepsilon, \mfa j>L_{K,\varepsilon}, n>N_{K,\varepsilon}.
\]
\end{Lem}
 Given this result, observe that the left-hand side of~\eqref{eq:ergodicity_K} without taking the limit is bounded by, writing $\sum_{j_1'\in\Z} = \sum_{|j_1-j_1'|\le L_{K,\varepsilon}} + \sum_{|j_1-j_1'|>L_{K,\varepsilon}}$ for each $j_1\in\Z$, 
\begin{multline}
CL_{K,\varepsilon}\sum_{j_1\in\Z}\pp{\frac{b_{n_1,j_1}\topp1}{b_{n_1}\topp1}}^2\sup_{j_1'\in\Z}\pp{\frac{b_{n_1,j_1'}\topp1}{b_{n_1}\topp1}}^2 + \frac1{(b_{n_1}\topp1)^4}\sum_{j_1,j_1'\in\Z}(b_{n_1,j_1}\topp1)^2 (b_{n_1,j_1'}\topp1)^2 \varepsilon \\
\le CL_{K,\varepsilon}\sup_{j_1'\in\Z}\pp{\frac{b_{n_1,j_1'}\topp1}{b_{n_1}\topp1}}^2 + \varepsilon,\label{eq:CLK}
\end{multline}
for all $n_2>N_{K,\varepsilon}$.
This and Lemma \ref{lem:sup_bnj} give~\eqref{eq:ergodicity_K} and hence the third condition of McLeish's central limit theorem \eqref{eq:McLeish3}. Therefore, the proof of Proposition~\ref{prop:HSHS_CLT} is completed. 
It remains to prove Lemma \ref{lem:cov_decay}. 

\begin{proof}[Proof of Lemma~\ref{lem:cov_decay}]
Introduce, for each $K\in\N, j\in\Z$, the event
\equh\label{eq:RjK}
R_{j,K}\topp1= \ccbb{\pp{\bigcup_{i\in\{-K+1,\dots,0\}}A_i\topp1} \cap \pp{\bigcup_{i'\in\{j-K+1,\dots,j\}}A_{i'}\topp1} = \emptyset}.
\eque
We have $\lim_{j\to\infty}\proba(R_{j,K}\topp1) = 1$ for all $K$. 
This comes from 
\begin{align*}
\proba(A_0\topp1\cap A_j\topp1\ne\emptyset) 
&\le \sum_{k\ge 0} \proba(-k\in A_0\topp1, -k\in A_j\topp1)\\
& \le \sum_{k\ge 0} q_{k}\topp1 q_{j+k}\topp1 \le \pp{\sum_{k\ge 0} (q_{k}\topp1)^2}^{1/2}\pp{\sum_{k\ge 0} (q_{j+k}\topp1)^2}^{1/2},
\end{align*}
and the fact that $\sum_{k\ge 0} (q_{j+k}\topp1)^2\to 0$ as $j\to\infty$.
We now write
\equh\label{eq:UU1}
\esp(U_{0,n,K}^2U_{j,n,K}^2) = \esp\pp{U_{0,n,K}^2U_{j,n,K}^2\ind_{R_{j,K}\topp1}} + \esp\pp{U_{0,n,K}^2U_{j,n,K}^2\ind_{(R_{j,K}\topp1)^c}}.
\eque
The second term on the right-hand side, by applying the Cauchy--Schwarz inequality twice and Lemma \ref{lem:Uapprox}, can be bounded uniformly in $n$ by $C\proba((R_{j,K}\topp1)^c)^{1/2}$, which goes to zero as $j\to\infty$ ($C$ is a positive constant). So, it suffices to show that the first term on the right-hand side above can be controlled to be arbitrarily close to $(\esp U_{0,n,K}^2)^2$ for $j,K$ large enough. 

For this purpose, the key idea is to decouple the underlying partitions in the first direction $U_{0,n,K}$ and $U_{j,n,K}$. Otherwise, notice that the two are dependent for all choices of $j$ and $K$. For the decoupling, first 
notice that the law of the partition in the first direction are determined by the law of those ancestral lines involved in the definition of $R_{j,K}\topp1$. 
To proceed we introduce a copy of $\{A_j\topp1\}_{j\in\Z}$, denoted by $\{\wt A_j\topp1\}_{j\in\Z}$, independent of the original two-dimensional Hammond--Sheffield model. Introduce the product partition $\wt G$ of $\Z^2$ as in the original model, but instead induced by $\{\wt A_j\topp1\}_{j\in\Z}$ and $\{A_j\topp2\}_{j\in\Z}$. Then, define $\{\wt X_\vvj\}_{\vvj\in\Z^2}$ as before  on  $\wt G$ by identical assignment rule. Define similarly $\wt X_\vvi^*, \wt X_{\vvi,K}^*, \wt U_{j,n}$ and $\wt U_{j,n,K}$ as before,  based on $\wt G$. These are identically distributed copies of the corresponding quantities of the original model. Define
\equh\label{eq:tildeRjK}
\wt R_{j,K}\topp1= \ccbb{\pp{\bigcup_{i\in\{-K+1,\dots,0\}}A_i\topp1} \cap \pp{\bigcup_{i'\in\{j-K+1,\dots,j\}}\wt A_{i'}\topp1} = \emptyset}.
\eque
We first remark that $\proba(R_{j,K}\topp1) = \proba(\wt R_{j,K}\topp1)$ for $j\ge K$ and 
\[
\left.\pp{U_{0,n,K},U_{j,n,K}}\mmid R_{j,K}\topp1\right.
\eqd \left.\pp{U_{0,n,K},\wt U_{j,n,K}}\mmid \wt R_{j,K}\topp1\right.,
\]
where each side is understood as the conditional distribution of a bivariate random vector.
Therefore, we have
\begin{align}
\esp\pp{U_{0,n,K}^2U_{j,n,K}^2\ind_{R_{j,K}\topp1}} & = \esp\pp{U_{0,n,K}^2\wt U_{j,n,K}^2\ind_{\wt R_{j,K}\topp1}}\nonumber\\
 & = \esp\pp{U_{0,n,K}^2\wt U_{j,n,K}^2} - \esp\pp{U_{0,n,K}^2\wt U_{j,n,K}^2\ind_{(\wt R_{j,K}\topp1)^c}} \nonumber\\
  & = \esp\pp{U_{0,n,K}^2\wt U_{0,n,K}^2} - \esp\pp{U_{0,n,K}^2\wt U_{j,n,K}^2\ind_{(\wt R_{j,K}\topp1)^c}}\label{eq:UU2}
\end{align}
where in the last expression above, again the second term above is bounded by $C\proba((\wt R_{j,K}\topp1)^c)^{1/2}$, uniformly in $n$, for some positive constant $C$.
To sum up, by \eqref{eq:UU1} and \eqref{eq:UU2} we arrive at the fact that there exists a constant $L_{K,\varepsilon}$ such that 
\equh\label{eq:UU3}
\abs{\esp\pp{U_{0,n,K}^2 U_{j,n,K}^2} - \esp\pp{U_{0,n,K}^2\wt U_{0,n,K}^2}}\le \frac\varepsilon2,\; \mfa j>L_{K,\varepsilon}, n\in\N.
\eque
Finally, we will prove that 
\equh\label{eq:cov_decay2}
\limn\cov\pp{U_{n,0,K}^2,\wt U_{n,0,K}^2}=0.
\eque
This and \eqref{eq:UU3} shall yield that there exists an integer $N_{K,\varepsilon}$ such that 
\[
\abs{\esp\pp{U_{0,n,K}^2\wt U_{0,n,K}^2} - (\esp U_{0,n,K}^2)^2} < \frac\varepsilon2, \mfa n\in N_{k,\varepsilon},
\]
and  complete the proof of the lemma.

It remains to show~\eqref{eq:cov_decay2}. 
We start by establishing a conditional central limit theorem for $U_{0,n,K}$, given the ancestral lines $A_{-K+1}\topp1,\dots,A_0\topp1$. We shall actually only need the random partition on $\{-K+1,\dots,0\}$, denoted by $G_K\topp 1$, induced by these ancestral lines. 
Recall the definition of $U_{0,j,K}$ in \eqref{eq:UjnK}. We have
\equh\label{eq:U0nK}
U_{0,n,K} = \frac1{b_n\topp 2}\summ j1n X_{(0,j),K}^* = \frac1{b_n\topp 2}\summ {j_2}1n\pp{X_{(0,j_2)}-\summ{j_1}1Kp_{j_1}\topp1X_{(-j_1,j_2)}} \equiv \frac1{b_n\topp 2}\summ j1n \mathbb X_{j},
\eque
where we introduce $\mathbb X_{j} = X_{(0,j),K}^*$ to simplify the notation. Note that $\mathbb X$ depends on $K$. 

Here we need the crucial remark that, given $G\topp 1_K$, $\{\X_i\}_{i\in\Z}$ is a generalized one-dimensional Hammond--Sheffield model, with a marginal law as the conditional law of $X_{\vv0,K}^*$ given $G\topp 1_K$, and hence with bounded support (Remark \ref{rem:boundedHS}). To see this, the second expression of $U_{0,n,K}$ in \eqref{eq:U0nK} is more convenient: by definition of the two-dimensional model, it suffices to examine the partition of $\{-K+1,\dots,0\}\times\N$. 
Recall that the product partition is obtained by Cartesian products. It then follows that $\X_i\equiv \X_j$ if $i\sim j$ with respect to the random partition $G\topp2$ of the second direction of the model, and otherwise $\X_i$ and $\X_j$ are i.i.d. Note that this observation remains true if we condition on $G_K\topp1$ first; the marginal law will depend on $G_K\topp1$, but remains bounded. 
Then, Proposition \ref{prop:generalizedHS} tells that 
\equh\label{eq:g_HS}
\left.\ccbb{\frac1{b_n\topp 2}\summ i1{\floor{nt}} \X_i}_{t\in[0,1]}\mmid G_K\topp1\right.\weakto \sigma_K\ccbb{\B^H_t}_{t\in[0,1]},
\eque
where  $H = \alpha_2+1/2$ and
\[
\sigma_K^2 = \frac{\esp(\X_1^2\mid G_K\topp1)}{\sum_{k\ge0}(q_k\topp2)^2} = \frac{\esp(X_{\vv0,K}^{*2}\mid G_K\topp1)}{\sum_{k\ge 0}(q_k\topp2)^2}.
\]
See Appendix \ref{sec:appendix} for our notations for conditional limit theorems.  We only need $t=1$ to deal with $U_{0,n,K}$ in the central limit theorem here, but for the proof of finite-dimensional distributions later, we shall need the above conditional functional central limit theorem.
In particular, \eqref{eq:g_HS} yields that
\[
 U_{0,n,K}\mid G_K\topp1 \weakto  \sigma_K \cdot \calN(0,1).
\]
Introduce similarly $\wt G_K\topp1$ based on $\{\wt A_{-K+1}\topp1,\dots,\wt A_0\topp1\}$. By the same approach described above, we can show that 
\equh\label{eq:CCLT_joint}
\left.\pp{U_{0,n,K},\wt U_{0,n,K}}\mmid G_K\topp1\right.,\, \wt G_K\topp1
\weakto \pp{\sigma_K Z, \wt \sigma_K \wt Z},
\eque
where  $Z$ and $\wt Z$ are two independent standard normal random variables and $\wt \sigma_K^2 = \esp(\wt X_{\vv0,K}^{*2}\mid \wt G_K\topp1)(\sum_{k\ge 0}(q_k\topp2)^2)^{-1}$. To establish the joint convergence, by the Cram\'er--Wold device it suffices to consider, for all $a,b\in\R$, 
\begin{align*}
\left.aU_{0,n,K}+b\wt U_{0,n,K}\mmid G_K\topp1, \wt G_K\topp1\right.&
\eqd \left.\frac1{b_n\topp2}\sum_{i=1}^n\pp{a X_{(0,i),K}^*+b {\wt X}_{(0,i),K}^*}\mmid G_K\topp1, \wt G_K\topp1\right. ,
\end{align*}
where ${\wt X}_{(0,i),K}^*$ is defined similarly as $X_{(0,i),K}^*$, and the two are  assumed to be conditionally independent given $\{A_j\topp2\}_{j\in\Z}$.
Again, given $G_K\topp1$ and $\wt G_K\topp1$,
the process
\[
\ccbb{\wb \X_i}_{i\in\N} \equiv \ccbb{a X_{(0,i),K}^*+b {\wt X}_{(0,i),K}^*}_{i\in\N}
\]
is a generalized one-dimensional Hammond--Sheffield model. The normalized partial sum $\summ i1n \X_i/b_n\topp2$ then converges to a normal distribution, with variance equal to 
\[
\var\pp{\wb\X_1\mmid G_K\topp1,\wt G_K\topp 1} = \esp\bb{\pp{a X_{(0,i),K}^*+b {\wt X}_{(0,i),K}^*}^2 \mmid G_K\topp1, \wt G_K\topp1}= a^2\sigma_{K}^2+b^2\wt \sigma_{K}^2. 
\]
Hence, \eqref{eq:CCLT_joint} follows as before by Proposition \ref{prop:generalizedHS}.

As a consequence of~\eqref{eq:CCLT_joint},  we arrive at
\[
\pp{U_{0,n,K},\wt U_{0,n,K}}
\weakto \pp{\sigma_{K}Z, \wt\sigma_{K}\wt Z},\mmas n\to\infty,
\]
where now $\sigma_K$ and $\wt\sigma_K$ are random variables, and all four random variables on the right-hand side are independent. By the boundedness of $\sigma_{K}, \wt \sigma_{K}$ and the uniform integrability of $U_{n,0,K}^4$ and $\wt U_{n,0,K}^4$, it follows that
\[
\limn\cov\pp{U_{n,0,K}^2,\wt U_{n,0,K}^2} = \cov\pp{\sigma_{K}^2Z^2,\wt\sigma_{K}^2\wt Z^2} = 0.
\]
This completes the proof of~\eqref{eq:cov_decay2}.
\end{proof}

\subsection{Proof of Theorem \ref{thm:HSproduct}}
\begin{proof}[Proof of convergence of finite-dimensional distributions]
We use Cram\'er--Wold device. For $m\in\N$, let $\lambda_1,\dots,\lambda_m\in\R$ and  $\vvt\topp1,\dots,\vvt\topp m\in[0,1]^2$ be fixed.
For $\vvn\in\N^2$, denote $\vvn_{\vvt\topp1},\dots,\vvn_{\vvt\topp m}\in\N^2$ as before in~\eqref{eq:nr} and to shorten the notation, the two coordinates of $\vvn_{\vvt\topp r}$ are denoted by $n_1(r)=\floor{n_1t_1\topp r}$ and $n_2(r)=\floor{n_2 t_2\topp r}$ respectively, $\vvt\topp r$ being fixed.
We have for all $\vvn \in\N^2$,
\begin{align*}
\frac1{b_\vvn}\summ r1m\lambda_rS_{\vvn_{\vv t\topp r}} 
& = \frac1{b_{n_1}\topp1}\sum_{j_1\in\Z}\summ r1m \lambda_r  b_{n_1(r),j_1}\topp1U_{j_1,n_2(r)},
\end{align*}
where $U_{j_1,n_2(r)}$ is defined as in \eqref{eq:Uj_1,n_2}.
One can show as before that 
\[
\ccbb{\summ r1m \lambda_rb_{n_1(r),j}\topp1 U_{j,n_2(r)}}_{j\in\Z}
\]
is a martingale-difference sequence with respect to $\{\calF_{j}\topp1\}_{j\in\Z}$. Therefore we apply the central limit theorem of McLeish as in Section~\ref{sec:HSHS_CLT}. The two conditions corresponding to~\eqref{eq:McLeish1} and~\eqref{eq:McLeish2} can be verified similarly as before. The third condition~\eqref{eq:McLeish3} becomes
\[
\lim_{n_1\to\infty}\frac1{(b_{n_1}\topp1)^2}\sum_{j\in\Z}\pp{\summ r1m \lambda_rb_{n_1(r),j_1}\topp1 U_{j,n_2(r)}}^2  =  \frac{\var\pp{\summ r1m \lambda_r \B_{\vvt\topp r}^{\vvH}}}{(\sum_{\vvk\ge\vv0}q_{\vvk}^2)^2} \mbox{ in probability}.
\]
This shall follow from
\equh\label{eq:McLeish3'}
\lim_{n_1\to\infty}\frac1{(b_{n_1}\topp1)^2}\sum_{j\in\Z}b_{n_1(r),j}\topp1 b_{n_1(r'),j}\topp1 U_{j,n_2(r)}U_{j,n_2(r')}  = \frac{\cov\spp{\B_{\vvt\topp r}^{\vvH},\B_{\vvt\topp{r'}}^{\vvH}}}{(\sum_{\vvk\ge\vv0} q_{\vvk}^2)^2} \mbox{ in probability,}
\eque
for all $r,r'\in\{1,\dots,m\}$. We do so again by computing the $L^2$-convergence. Remark first that
\begin{align*}
\esp & \pp{\frac1{(b_{n_1}\topp1)^2}\sum_{j\in\Z}b_{n_1(r),j}\topp1 b_{n_1(r'),j}\topp1U_{j,n_2(r)}U_{j,n_2(r')}} \\
& = \frac1{(b_{n_1}\topp1)^2}\sum_{j_1\in\Z}b_{n_1(r),j_1}\topp 1b_{n_1(r'),j_1}\topp 1\frac1{(b_{n_2}\topp2)^2}\sum_{j_2\in\Z}b_{n_2(r),j_2}\topp 2 b_{n_2(r'),j_2}\topp 2\var (X_{\vv0}^{**})\\
& \sim \cov\pp{\B_{\vvt\topp r}^\vvH,\B_{\vvt\topp{r'}}^\vvH} \var (X_{\vv0}^{**})\mmas \vvn\to\infty,
\end{align*}
where $X_\vv0^{**}$ is defined as in \eqref{eq:defX**} and the asymptotic follows from the identity 
\begin{align*}
 \sum_{j\in\Z}b_{n(r),j}\topp m  b_{n(r'),j}\topp m
&=\frac12 \bb{(b_{n(r)}\topp m)^2+ (b_{n(r')}\topp m)^2 - |b_{n(r)}\topp m -  b_{n(r')}\topp m|^2 }\\
&=\frac12 \bb{ (b_{n(r)}\topp m)^2+ (b_{n(r')}\topp m)^2 - (b_{|n(r)-n(r')|}\topp m)^2},\quad m=1,2,
\end{align*}
and \eqref{eq:asymb_n}.
Therefore, to show~\eqref{eq:McLeish3'}, it suffices to prove, as a counterpart of \eqref{eq:ergodicity_K},
\begin{align*}
&\lim_{\vvn\to\infty}
\frac1{(b_{n_1}\topp1)^4}\sum_{j,j'\in\Z}b_{n_1(r),j}\topp1 b_{n_1(r'),j}\topp1b_{n_1(r),j'}\topp1 b_{n_1(r'),j'}\topp1\\
&\hspace{30pt}\cdot\cov\pp{U_{j,n_2(r),K}U_{j,n_2(r'),K},U_{j',n_2(r),K}U_{j',n_2(r'),K}} = 0,
\end{align*}
which, as in~\eqref{eq:CLK}, shall follow from the following lemma. The proof of convergence of finite-dimensional distributions is thus completed.
\end{proof}
\begin{Lem}\label{lem:cov_decay2}
For all $K\in\N,\varepsilon>0$, there exists $L_{K,\varepsilon}, N_{K,\varepsilon}$, such that
\[
\abs{
\cov\pp{U_{0,n_2(r),K}U_{0,n_2(r'),K},U_{j,n_2(r),K}U_{j,n_2(r'),K}}}<\varepsilon, \mfa j>L_{K,\varepsilon}, n_2>N_{K,\varepsilon}.
\]
\end{Lem}
\begin{proof}
By the same idea as in the proof of Lemma~\ref{lem:cov_decay}, it suffices to show
\[
\lim_{n_2\to\infty}\cov\pp{U_{0,n_2(r),K}U_{0,n_2(r'),K},\wt U_{0,n_2(r),K}\wt U_{0,n_2(r'),K}} = 0.
\]
As a consequence of \eqref{eq:g_HS}, instead of~\eqref{eq:CCLT_joint} we now have
\begin{multline*}
\left.\pp{U_{0,n_2(r),K},U_{0,n_2(r'),K},\wt U_{0,n_2(r),K},\wt U_{0,n_2(r'),K}}\mmid{G_K\topp1, \wt G_K\topp1}\right.
\\
\weakto \pp{\sigma_{K}\B_{t_2\topp r}^{H_2},\sigma_{K}\B_{t_2\topp {r'}}^{H_2}, \wt\sigma_{K}\wt \B_{t_2\topp {r}}^{H_2},\wt\sigma_{K}\wt \B_{t_2\topp {r'}}^{H_2}},
\end{multline*}
where $\B^{H_2}$ and $\wt \B^{H_2}$ are i.i.d.~copies of fractional Brownian motion with Hurst index $H_2=\alpha_2+1/2$,  $\sigma_{K}$ and $\wt\sigma_{K}$ are as before. 
This completes the proof.
\end{proof}

\begin{proof}[Proof of tightness]
Again, applying Bickel--Wichura's criterion \citep{bickel71convergence} and using \eqref{eq:asympb_n}, the tightness will follow from the existence of a real $\gamma>1$ such that
\eqnh
\esp\pp{\frac{S_{\vvm}}{b_\vvn}}^2 \leq C \prodd q12 \pp{\frac{m_q}{n_q}}^{\gamma},\;\text{ for all }\vvm\le\vvn.
\eqne
Let $\delta>0$ be such that $\gamma = 2\min(\alpha_1,\alpha_2)+1-\delta>1$. Using the representation of $S_\vvm$ as in \eqref{eq:Sn/bn} and the representation of $U_{0,n}$ as in \eqref{eq:U_n}, applying Burkholder's inequality twice, we get
\[
\esp\pp{\frac{S_{\vvm}}{b_\vvn}}^2
\le C \pp{\frac{b_{m_1}\topp1}{b_{n_1}\topp1}}^2 \pp{\frac{b_{m_2}\topp2}{b_{n_2}\topp2}}^2\|X_\vv0^{**}\|_{2}^2 \le C \prodd q12 \pp{\frac{m_q}{n_q}}^{2\alpha_q+1}\frac{L_q(n_q)}{L_q(m_q)}.
\]
We obtain the desired result by Potter's bound ${L_q(n_q)}/{L_q(m_q)}\le C\spp{{m_q}/{n_q}}^{-\delta}$.
\end{proof}


\section{Combining Hammond--Sheffield model and Karlin model}\label{sec:combined}
In this section, we combine a one-dimensional Hammond--Sheffield model and  a one-dimensional randomized Karlin model together, and show that the combined model converges weakly to a fractional Brownian sheet with Hurst indices $H_1\in(1/2,1)$ and $H_2\in(0,1/2)$.

\subsection{Model and main result}
Consider two random partitions from the one-dimensional Hammond--Sheffield model and the randomized Karlin model, respectively. Assume the two random partitions are independent. Namely, let $G\topp1 = G(E,V)$ be the underlying random forest structure of the Hammond--Sheffield model generated by a distribution $\mu_1$, and let $\{A_j\topp1\}_{j\in\Z}$ be the associated ancestral lines. Let $\{Y_j\topp2\}_{j\in\N}$ be i.i.d.~random variables with common distribution $\mu_2$. Suppose $\mu_i, i=1,2$ are probability measures on $\N$ satisfying \eqref{eq:RVHS} and \eqref{eq:RV} respectively with $\alpha_1\in(0,1/2)$ and $\alpha_2\in(0,1)$.
Assume $G\topp1$ and $\{Y_j\topp2\}_{j\in\N}$ are independent. 
Now, consider the product of the two random partitions.
This is the random partition of $\Z\times\N$ determined by 
\[
\vvi\sim\vvj \;\mbox{ if and only if }\;  A_{i_1}\topp1\cap A_{j_1}\topp1\neq\emptyset \mand Y_{i_2}\topp2 = Y_{j_2}\topp2.
\]

Next, given the partition, we apply the identical assignment rule in the first direction, and the alternating assignment rule in the second  (see Figure~\ref{fig:1}, right).
Given a collection of components determined by $G\topp1$ and $\{Y_n\topp2\}_{n\in\N}$, we assign values $X_\calC = \{X_\vvj\}_{\vvj\in\calC}$ as follows. Let $\{\epsilon_\calC\}_{\calC}$ be a collection of i.i.d.~random variables taking values in $\{-1,1\}$ with equal probabilities, indexed by different components $\calC$. For each $\calC$ fixed, express this as
\[
\calC = \calC\topp1\times \{j_\ell\topp2\}_{\ell\in\N} \qmwith 1\le j_1\topp2\le j_2\topp2\le\cdots,
\]
and set
\[
X_\vvj = (-1)^{\ell+1}\epsilon_\calC, \quad \mbox{ for } \vvj = (j_1,j_2)\in\calC,\; j_2 = j\topp2_\ell.
\]
The random field $\{X_\vvj\}_{\vvj\in\Z\times\N}$ constructed this way is referred to as the two-dimensional combined model.
The main result of this section is the following invariance principle for $S_\vvn = \sum_{\vvi\in[\vv1,\vvn]}X_\vvi$. 
\begin{Thm}\label{thm:WIPcombined}
For the two-dimensional combined model with $\alpha_1\in(0,1/2)$, $\alpha_2\in(0,1)$, 
and slowly varying functions $L_1$, $L_2$ respectively,
\eqnh
\ccbb{\frac{S_{\floor{\vvn\cdot\vvt}}}{n_1^{H_1}n_2^{H_2}L_1(n_1)^{-1}L_2(n_2)^{1/2}}}_{\vvt\in[0,1]^2}\weakto \sigma_\vv\alpha \ccbb{\B^{\vv H}_\vvt}_{\vvt\in[0,1]^2}
\eqne
in $D([0,1]^2)$, as $\vvn\to\infty$, with $H_1 = \alpha_1+1/2\in(1/2,1)$, $H_2 = \alpha_2/2 \in(0,1/2)$, and
\[
\sigma_\vv\alpha^2 = \frac{C_{\alpha_1}\Gamma(1-\alpha_2)2^{\alpha_2-1}}{\sum_{j\ge 0}(q_j\topp1)^2},
\]
for $C_{\alpha_1}$ defined in~\eqref{eq:Calpha} and $q_j\topp1=\proba(0\in A_j\topp1)$, $j\in\Z$.
\end{Thm}

\subsection{Proof of Theorem~\ref{thm:WIPcombined}}

The proof follows the same strategy as for  the two-dimensional Hammond-Sheffield model in Sections \ref{sec:representation} and \ref{sec:HSHS_CLT}.
We first introduce the sequence $\{X_\vvj^*\}_{j\in\Z}$ defined as in \eqref{eq:X*} by 
\[
 X_\vvj^*=X_\vvj-\esp(X_\vvj \mid \calF_{j_1}\topp1),\quad \vvj\in\Z,
\]
where $\calF_j\topp1=\sigma\{X_\vvi\mid i_1<j,\, i_2\in\N\}$. 
Note that, to draw a parallel with Sections \ref{sec:representation} and~\ref{sec:HSHS_CLT}, we keep the same notation but the variables $X_\vvj^*$ here are different from the preceding section since the dependence in the second direction is given by a partition from the Karlin model.
Nevertheless, for any $j_2\in\N$, the sequence $\{X_\vvj\}_{j_1\in\Z}$ is a martingale-difference sequence with respect to $\{\calF_{j_1}\topp1\}_{j_1\in \Z}$. So, Lemma~\ref{lem:projection} remains valid here (the proof is exactly the same) and we thus have
\[
 S_\vvn=\sum_{j_1\in\Z} b_{n_1,j_1}\topp1\sum_{j_2=1}^{n_2}X_\vvj^*,
\]
with $b_{n,j}\topp1 = \sum_{k=1}^n q_{k-j}\topp1$ defined as before.
Recall from \eqref{eq:asymb_n} that
\[
( b_{n}\topp1)^2=\sum_{j\in\Z}(b_{n,j}\topp1)^2  \sim C_{\alpha_1} n^{2\alpha_1 +1}L_1(n)^{-2},\mmas n\to\infty.
\]
Because of the alternating assignment rule in the second direction, we need to consider the number
$\wt K_{n}\topp2=\sum_{i=1}^\infty K_{n,2i-1}\topp2$ of odd-occupancy boxes, that is the number of values appearing an odd number of times among $\{Y_1\topp2,\ldots, Y_n\topp2\}$.
Recall from Lemma~\ref{lem:K} that
\[
 (a_n\topp2)^2 =\esp\wt K_{n}\topp2 \sim  \Gamma(1-\alpha_2) 2^{\alpha_2-1}n^{\alpha_2}L_2(n), \mmas n\to\infty.
\]
This time, for all $\vvn\in\N^2$, we can write
\[
 \frac{S_\vvn}{b_{n_1}\topp1 a_{n_2}\topp2}=\frac{1}{b_{n_1}\topp1}\sum_{j\in\Z}b_{n_1,j_1}\topp1 U_{j_1,n_2} \quad\text{ with }\quad  U_{j_1,n_2}=\frac{1}{a_{n_2}\topp2}\sum_{i=1}^{n_2} X_\vvj^*.
\]
This is the counterpart of Proposition \ref{prop:representation}, representing the normalized partial sum of interest as a weighted linear process with stationary martingale-difference innovations.

We then introduce, for all $K\ge 1$,  the approximations
\[
X_{\vvj,K}^*=X_\vvj-\sum_{k=1}^K p_k\topp1 X_{(j_1-k,j_2)}\quad\text{ and }\quad U_{j_1,n,K}=\frac{1}{a_{n}\topp2} \sum_{j_2=1}^n X_{\vvj,K}^*,
\]
for all $j_1\in\Z$, $j_2\in\N$, $n\in\N$. 

\begin{proof}[Proof of convergence of finite-dimensional distributions] This can be done as in Section~\ref{sec:HSHS_CLT} by the use of Cram\'er--Wold device and McLeish's theorem \citep{mcleish74dependent}. For this purpose, we only need to show that the conclusions of Lemmas~\ref{lem:Umoment},~\ref{lem:Uapprox},~\ref{lem:cov_decay}, and~\ref{lem:cov_decay2} are still valid with respect to the newly defined  $X_\vvj^*$, $X_{\vvj,K}^*$, $U_{j_1,n_2}$ and $U_{j_1,n_2,K}$. For the sake of convenience, we restate  Lemmas \ref{lem:Umoment} and \ref{lem:Uapprox} in Lemma \ref{lem:U} below, and restate Lemma \ref{lem:cov_decay} in Lemma \ref{lem:cov_decay3} below. 
The core arguments of Lemma \ref{lem:cov_decay2} are all in Lemma \ref{lem:cov_decay} and we therefore omit the proof.
\end{proof}

\begin{Lem}\label{lem:U}(i) For all $n\in\N$, $\esp U_{0,n}^2 = (\sum_{\vvk\ge\vv0}q_\vvk^2)\inv<\infty$. 

\noindent (ii) For all $p\in\N$ and $K\in\N$, $\esp U_{0,n}^{2p}$ and $\esp U_{0,n,K}^{2p}$ are uniformly bounded.

\noindent (iii) For all $p\ge 1$, $\lim_{K\to\infty}\sup_{n\in\N}\esp|U_{0,n}-U_{0,n,K}|^{2p} = 0$. 
\end{Lem}
\begin{proof}
Denoting by $\calG\topp1$ and $\calY\topp2$ the $\sigma$-fields generated by $G\topp1$ and $\{Y_j\topp2\}_{j\in\N}$ respectively,  as for \eqref{eq:S_nK_n}, we see that for all $n\in\N$, $j_1\in\N$,
\[
\left.\sum_{j_2=1}^nX_\vvj^* \mmid \calG\topp1,\calY\topp2 \right.\eqd \left.\sum_{i=1}^{\wt K_n\topp2} \epsilon_i'\mmid  \calG\topp1,\right.
\]
where the random variables $\{\epsilon_i'\}_{i\in\N}$ are conditionally independent given $\calG\topp1$, independent of $\calY\topp2$, and for all $i \in\N$, the conditional distribution $\epsilon_i'\mid \calG\topp1$ is the same as the conditional distribution $X_\vv0^*\mid  \calG\topp1$. Note in the identity above, without the conditioning on $\calG\topp1$, the $\{\epsilon'_i\}_{i\in\N}$ on the right-hand side are no longer independent. 
We can thus write
\begin{align*}
 \esp\pp{ U_{0,n}^2\mmid \calG\topp1, \calY\topp2}
&=\frac{1}{(a_n\topp2)^2} \esp\bb{ \pp{\sum_{j_2=1}^nX_\vvj^*}^2\mmid \calG\topp1, \calY\topp2}\\
&=\frac{1}{(a_n\topp2)^2} \esp\bb{ \pp{\sum_{i=1}^{\wt K_n\topp2} \epsilon_i'}^2 \mmid \calG\topp1,\calY\topp2} =\frac{\wt K_n\topp2}{(a_n\topp2)^2} \esp\pp{\epsilon_0'^2\mid\calG\topp1},
\end{align*}
Thus
\[
  \esp\pp{ U_{0,n}^2}=\var(X_\vv0^*)<\infty,\; \text{uniformly in }n.
\]
This proves the first part. 

For the second part,  for all $p\ge 1$, by Burkholder's inequality we have
\[
  \esp\pp{ U_{0,n}^{2p}\mmid \calG\topp1,\calY\topp2}
\le C_p \pp{\frac{\wt K_n\topp2}{(a_n\topp2)^2}}^p \esp(X_\vv0^{*2p}\mid\calG\topp1).
\]
Note that $\esp\spp{{\wt K_n\topp2}/{(a_n\topp2)^2}}^p$ is uniformly bounded by Lemma~\ref{lem:K}. Similarly, 
\[
   \esp\pp{ U_{0,n,K}^{2p}\mmid\calG\topp1, \calY\topp2}
\le C_p \pp{\frac{\wt K_n\topp2}{(a_n\topp2)^2}}^p \esp(X_{\vv0,K}^{*2p}\mid\calG\topp1).
\]

For the third part, we have
\[
  \esp\pp{ |U_{0,n,K}-U_{0,n}|^{2p}\mmid \calG\topp1,\calY\topp2}
\le C_p \pp{\frac{\wt K_n\topp2}{(a_n\topp2)^2}}^p \esp(|X_{\vv0,K}^*-X_\vv0^*|^{2p}\mid\calG\topp1)\to 0,\
\]
as $K\to\infty$.
\end{proof}

\begin{Lem}\label{lem:cov_decay3}
For all $K\in\N,\varepsilon>0$, there exists integers $L_{K,\varepsilon}, N_{K,\varepsilon}$, such that
\[
\abs{
\cov(U_{0,n,K}^2,U_{j,n,K}^2)}<\varepsilon, \mfa j>L_{K,\varepsilon}, n>N_{K,\varepsilon}.
\]
\end{Lem}

\begin{proof}
To proceed we introduce a copy of $\{A_j\topp1\}_{j\in\Z}$, denoted by $\{\wt A_j\topp1\}_{j \in\Z}$, independent of the original model and we defined a new field $\{\wt X_\vvj\}_{\vvj\in\Z}$ based on the combined model involving $\{\wt A_j\topp1\}_{j\in\Z}$ and the same $\{Y_j\topp2\}_{j\in\N}$ as the original model. Then $\wt X_\vvj$, $\wt X_\vvj^*$, $\wt X_{\vvj,K}^*$, $\wt U_{j,n}$, and $\wt U_{j,n,K}$ are defined as the corresponding statistics of the combined model based on $\{\wt A_j\topp1\}_{j\in\Z}$ and  $\{Y_j\topp2\}_{j\in\N}$.  
In particular, these random variables are identically distributed as the variables $X_\vvj, X_\vvj^*, X_{\vvj,K}^*, U_{0,n}$ and $U_{0,n,K}$, respectively, and they are conditionally independent from them given $\calY\topp2$. 
As in the proof of Lemma~\ref{lem:cov_decay}, observe that
\[
 (U_{0,n,K},U_{j,n,K})\mid R_{j,K}\topp1 \eqd \left.\pp{U_{0,n,K},\wt U_{j,n,K}}\mmid \wt R_{j,K}\topp1,\right.
\]
for $R_{j,K}\topp1$ and $\wt R_{j,K}\topp1$ defined as in \eqref{eq:RjK} and \eqref{eq:tildeRjK}. Therefore we see that to prove the desired result it suffices to show that for all $K\ge 1$,
\eqnh
\lim_{n\to \infty}\cov\pp{U_{0,n,K}^2,\wt U_{0,n,K}^2}=0,
\eqne
corresponding to \eqref{eq:cov_decay2}.
Let $G_K\topp1$ be the random partition of $\{-K+1,\ldots,0\}$ induced by $G\topp1$ and note that
\[
 U_{0,n,K}\mid G_K\topp1 \eqd \left.\frac{1}{a_{n_2}\topp2} \sum_{i=1}^n  X_{(0,i),K}^*\mmid G_K\topp1 \right. \equiv
 \left.\frac{1}{a_{n_2}\topp2} \sum_{i=1}^n  \X_i\mmid G_K\topp1 \right.,
\]
where again we write $\X_i = X_{(0,i),K}^*$ for the sake of simplicity. 
Here, conditionally given $G_K\topp1$, $\{\X_i\}_{i\in\N}$ is a generalized one-dimensional randomized Karlin model. 
Indeed, for $i\not\sim j$ with respect to $\{Y_j\topp 2\}_{j\in\N}$, 
the random variables $\X_i$ and $\X_j$ are conditionally independent given $G_K\topp1$ and $\calY\topp2$, and for $i\sim j$, 
letting $\ell$ denote the number of integers in the component between $i$ and $j$ (say $i<j$ without loss of generality, so $\ell = \{k:i<k<j, Y_k\topp 2 = Y_i\topp 2 = Y_j\topp2\}$), we have $\X_i = (-1)^{\ell+1}\X_j$ given $G_K\topp 1$ and $\calY\topp2$.

Similarly, let $\wt G_K\topp1$ be the random partition of $\{-K+1,\ldots,0\}$ induced by $\wt G\topp1$. 
Then for all $a,b\in\R$,
\[
\left. a U_{0,n,K}+b\wt U_{0,n,K}\mmid G_K\topp1, \wt G_K\topp1\right.
\eqd \left.\frac{1}{a_{n_2}\topp2} \sum_{i=1}^n \pp{a X_{(0,i),K}^*+b\wt X_{(0,i),K}^*}\mmid G_K\topp1, \wt G_K\topp1\right.,
\]
where
\[
\ccbb{\wb \X_i}_{i\in\N} \equiv \ccbb{a X_{(0,i),K}^*+b\wt X_{(0,i),K}^*}_{i\in\N},
\]
given $G_K\topp 1$ and $\wt G_K\topp 1$, this time is a generalized one-dimensional randomized Karlin model. Since $\{\wb\X_i\}_{i\in\N}$ has  bounded and centered marginal distribution, we can thus apply Proposition \ref{prop:Karlin1d} for generalized one-dimensional randomized Karlin model. The variance of the limit normal distribution is then
\[\esp\pp{ \wb \X_1^2\mmid G_K\topp1,\wt G_K\topp1} = 
a^2\var\pp{X_{\vv0,K}^*\mmid G_K\topp1}+b^2\var\pp{\wt X_{\vv0,K}^*\mmid \wt G_K\topp1} 
= a^2\sigma_K + b^2\wt \sigma_K.
\]
It follows that, by the Cram\'er--Wold device, 
\[
\pp{U_{0,n,K},\wt U_{0,n,K}} \weakto \pp{\sigma_K Z, \wt \sigma_K \wt Z},
\]
where $Z$ and $\wt Z$ are standard normal random variables and the four random variables are independent.
To conclude, we deduce that $\lim_{n\to\infty} \cov(U_{0,n,K}^2,\wt U_{0,n,K}^2)=\cov(\sigma_K^2 Z^2, \wt \sigma_K^2 \wt Z^2)=0$, which imply the desired result.
\end{proof}

\begin{proof}[Proof of tightness]
Again, we proceed using Bickel--Wichura's criterion \citep{bickel71convergence}. Observe that for all $n\in\N$, $\{U_{j,n}\}_{j\in\Z}$ is a martingale-difference sequence with respect to $\{\calF_{j}\topp1\}_{j\in\Z}$. By Burkholder's inequality, for all $p\ge 1$, for all $\vvn,\vvm\in\N^2$,
\begin{align*}
 \esp\pp{\frac{S_\vvm}{b_{n_1}\topp1 a_{n_2}\topp2}}^{2p}
&\le C_p\pp{\frac{a_{m_2}\topp2}{b_{n_1}\topp1 a_{n_2}\topp2}}^{2p} \esp\pp{\sum_{j\in\Z} (b_{m_1}\topp1)^2 U_{j,m_2}^2}^{p}\\
&\le C_p\pp{\frac{b_{m_1}\topp1}{b_{n_1}\topp1 }}^{2p}\pp{\frac{a_{m_2}\topp2}{a_{n_2}\topp2}}^{2p}
\esp U_{0,n}^{2p}.
\end{align*}
Using that $\esp U_{0,n}^{2p}$ is bounded uniformly in $n$ and that 
\[
 \pp{\frac{b_{m_1}\topp1}{b_{n_1}\topp1 }}^{2p}\pp{\frac{a_{m_2}\topp2}{a_{n_2}\topp2}}^{2p}\sim \frac{m_1^{H_1}L_1^{-1}(m_1)}{n_1^{H_1}L_1^{-1}(n_1)}\frac{m_2^{H_2}L_2(m_2)}{n_2^{H_2}L_2(n_2)},\mmas \vvn\to\infty,
\]
we can conclude as for the other models, dealing with the slowly varying functions by using Potter's bound. 
\end{proof}

\begin{Rem}\label{rem:high_dimension}
As we have seen, the proof follows the same structure as for the two-dimensional Hammond--Sheffield model. In fact, our models have their natural generalizations to high dimensions ($d\ge 2$), and the proof will follow the same strategy. The generalization of the model to high dimensions, based on independent random partitions and assignment rules in different directions, is  intuitively obvious. However, it is notationally heavy to introduce. We only briefly explain how the proof would go. If in all directions the random partition is the same as the one in the one-dimensional Karlin model, then the same proof as Theorem \ref{thm:1}, by first conditioning on the partition, shall work. If at least in one direction, say the first, the random partition and assignment rule are the ones of the one-dimensional Hammond--Sheffield model,  then the same strategy as in two-dimensional Hammond--Sheffield model and the combined model shall work, by first writing
\[
\frac1{b_\vvn} S_{n_1,\dots,n_d} = \frac1{b_{n_1}\topp1}\sum_{j_1\in\Z}b_{j_1,n_1}\topp1 U_{j_1,n_2,\dots,n_d},
\]
with $\{U_{j_1,n_2,\dots,n_d}\}_{j_1\in\Z}$ a stationary sequence of martingale differences. The analysis of this martingale-difference sequence shall need results for generalized $(d-1)$-dimensional models (to be defined properly first). To complete the details of this strategy  would require an induction argument. 
\end{Rem}


\appendix
\section{Conditional convergence}\label{sec:appendix}
We follow the notations of \citet[Chapter 5]{kallenberg97foundations}. Let $(\Omega,\calA,\proba)$ be a probability space, $(S,\calS)$ be a Borel space and $(T,\calT)$ be a measurable space. Let $\xi,\eta$ be two random elements in $S,T$ respectively. A regular conditional distribution of $\xi$ given $\eta$ is defined as a random measure $\nu$ of the form
\[
\nu(\eta,B) = \proba(\xi\in B\mid\sigma(\eta)), \mbox{ almost surely, } B\in\calS,
\]
where $\nu$ is a probability kernel from $(T,\calT)$ to $(S,\calS)$: $\nu(\cdot,B)$ is $\calT$-measurable for all $B\in\calS$, and $\nu(t,\cdot)$ is a probability measure on $(S,\calS)$ for all $t\in T$. Under the previous regularity assumption on the space $(S,\calS)$ and $(T,\calT)$, such a probability kernel $\nu$ exists, and is unique almost everywhere $\proba\circ\eta\inv$ \citep[Theorem 5.3]{kallenberg97foundations}. Furthermore, for all measurable function $f$ on $(S\times T,\calS\times\calT)$, with $\esp |f(\xi,\eta)|<\infty$, 
\[
\esp(f(\xi,\eta)\mid\sigma(\eta)) = \int \nu(\eta,d s)f(s,\eta), \mbox{ almost surely.}
\]
See for example \citep[Theorem 5.4]{kallenberg97foundations}.

Some of our results are in the form of  conditional (functional) limit theorems for the random field given underlying the random partition. The random partition, denoted by $\eta$ here, and the random field $\{X_\vvi\}_{\vvi\in\Z^d}, d\in\N$ are defined on a common probability space $(\Omega,\calA,\proba)$. Let $\indn Z$ be a sequence of real-valued random variables (the normalized partial sum with appropriate normalization) in the same probability space.
Then, let $\nu_n(\eta,\cdot)$ denote the regular conditional distribution of $Z_n$ given $\eta$.  
With $\calG = \sigma(\eta)$, we write for some $\calG$-measurable random variable $V$ (possibly a constant), 
\[
Z_n\mid\calG\weakto V\cdot \calN(0,1), 
\]
if $\nu_n(\eta(\omega),\cdot)$ as $n\to\infty$ converges to the standard normal distribution multiplied by $V(\omega)$ almost surely. That is, for all bounded continuous functions $h:\R\to\R$, 
\[
\limn \int h(z)\nu_n(\eta(\omega),dz) = \int h(z)\frac1{\sqrt{2\pi}}e^{-z^2/2}dz\cdot V(\omega), \mbox{ almost surely.}
\]
In this case we say that the conditional central limit theorem holds.

The conditional functional central limit theorem is interpreted in a similar way. Let $\calZ = \{Z(t)\}_{t\in T}$ and $\calZ_n = \{Z_n(t)\}_{t\in T}$, $n\in\N$ with $T = [0,1]^d$, $d\in\N$, be real-valued  stochastic processes in $D(T)$ equipped with the Skorohod topology, defined in the same probability space. We write
\[
\ccbb{\calZ_n(t)}_{t\in T}\mid \calG\weakto V\cdot \ccbb{\calZ(t)}_{t\in T},
\]
if, 
letting $\nu_n(\eta,\cdot)$ denote this time the regular conditional distribution of $\calZ_n$ given $\eta$ and $\mu_\calZ$ denote the distribution of $\calZ$, both as probability measures on $D(T)$, 
for all  bounded and continuous function $h$ from $D(T)$ to $\R$,
\[
\limn \int_{D(T)} h(\zeta)\nu_n(\eta(\omega),d\zeta) = \int_{D(T)} h(\zeta)\mu_\calZ(d\zeta)\cdot V(\omega)\; \mbox{ almost surely.}
\]


\def\cprime{$'$} \def\polhk#1{\setbox0=\hbox{#1}{\ooalign{\hidewidth
  \lower1.5ex\hbox{`}\hidewidth\crcr\unhbox0}}}
  \def\polhk#1{\setbox0=\hbox{#1}{\ooalign{\hidewidth
  \lower1.5ex\hbox{`}\hidewidth\crcr\unhbox0}}}

\end{document}